\documentclass{imamci}
\usepackage{mathtools}
\usepackage{amsthm}
\usepackage[colorlinks=true, urlcolor=blue, citecolor=blue, linkcolor=blue]{hyperref}
\usepackage{graphicx,psfrag}
\usepackage{caption}
\usepackage{pgfplots}                                                      
\usepackage{amssymb}
\usepackage{amsmath} 
\usepackage{algorithm}
\usepackage{algorithmic}
\usepackage[numbers]{natbib}
\newcommand{\range}{\operatorname{rge}}
\newcommand*{\bigO}{\scalebox{1.5}{\ensuremath{\circ}}}

\newcommand{\Co}{\operatorname{co}}
\newcommand{\ch}{\overline{\operatorname{co}}}
\newcommand{\0}{{\bf 0}}
\newcommand{\1}{{\bf 1}}
\newcommand*{\tr}{^{\mkern-1.5mu\mathsf{T}}}

\newcommand{\R}{{\mathbb{R}}}

\newcommand{\Spn}{{\mathcal{S}^n_+}}
\newcommand{\Dpn}{{\mathcal{D}^n_+}}

\newcommand{\Dpm}{{\mathcal{D}^{m}_+}}

\newcommand{\q}{\operatorname{q}}
\newcommand{\Sign}{\operatorname{sign}}
\newcommand{\dom}{\operatorname{dom}}
\newcommand{\Id}{\mathrm{I}}
\newcommand{\Int}{\mathrm{Int}\,}
\newcommand{\spec}{\mathrm{spec}}
\newcommand{\Tr}{\operatorname{trace}}
\newcommand{\minimize}{\operatorname{minimize}}
\newcommand{\He}{\operatorname{He}}
\newcommand{\Diag}{\operatorname{diag}}

\newcommand{\eg}{{e.g.}}
\newcommand{\ie}{{i.e.}}

\newcommand*{\QEDB}{\hfill\ensuremath{\square}}%

\jno{dnz003}

\usepackage{natbib}

\usepackage{graphicx}
\usepackage{amsmath,amsthm,bm,mathrsfs}
\usepackage{textcomp}

\renewenvironment{proof}[1][Proof]{\noindent\textit{#1. } }{\hfill$\square$}

 \newtheoremstyle{theorem}{6pt}{6pt}{\rm}{}{\sffamily}{ }{ }{}
 \theoremstyle{theorem}
\newtheorem{theorem}{\sc Theorem}[section]

 \newtheoremstyle{lemma}{6pt}{6pt}{\rm}{}{\sffamily}{ }{ }{}
 \theoremstyle{lemma}
 \newtheorem{lemma}{\sc Lemma}[section]

\newtheoremstyle{case}{6pt}{6pt}{\rm}{}{\sffamily}{. }{ }{}
 \theoremstyle{case}

 \newtheoremstyle{statement}{6pt}{6pt}{\rm}{}{\sffamily}{ }{ }{}
\theoremstyle{statement}

 \newtheoremstyle{corollary}{6pt}{6pt}{\rm}{}{\sffamily}{ }{ }{}
 \theoremstyle{corollary}

  \newtheoremstyle{definition}{6pt}{6pt}{\rm}{}{\sffamily}{ }{ }{}
 \theoremstyle{definition}
 \newtheorem{definition}{\sc Definition}[section]

\newtheoremstyle{example}{6pt}{6pt}{\rm}{}{\sffamily}{ }{ }{}
\theoremstyle{example}
\newtheorem{example}[theorem]{\sc Example}

\newtheorem{problem}[theorem]{Problem}

\newtheoremstyle{remark}{6pt}{6pt}{\rm}{}{\sffamily}{ }{ }{}
\theoremstyle{remark}
\newtheorem{remark}{\sc Remark}[section]

\newtheoremstyle{approximation}{6pt}{6pt}{\rm}{}{\sffamily}{ }{ }{}
\theoremstyle{approximation}

\newtheoremstyle{scheme}{6pt}{6pt}{\rm}{}{\sffamily}{ }{ }{}
\theoremstyle{scheme}

\newtheoremstyle{Algorithm}{6pt}{6pt}{\rm}{}{\sffamily}{ }{ }{}
\theoremstyle{Algorithm}

\newtheoremstyle{Assumption}{6pt}{6pt}{\rm}{}{\sffamily}{ }{ }{}
\theoremstyle{Assumption}

\newtheoremstyle{proposition}{6pt}{6pt}{\rm}{}{\sffamily}{ }{ }{}
\theoremstyle{proposition}
\newtheorem{proposition}{\sc Proposition}[section]

\newtheoremstyle{hypo}{6pt}{6pt}{\rm}{}{\sffamily}{ }{ }{}
 \theoremstyle{hypo}

  \newtheoremstyle{Step}{6pt}{6pt}{\rm}{}{}{ }{ }{}
 \theoremstyle{Step}

\numberwithin{equation}{section}

\begin{document}

\newcommand{\mytitle}{On Sensor Quantization in Linear Control Systems: Krasovskii Solutions meet Semidefinite Programming}
\title{\mytitle}
\author{ {\sc Francesco Ferrante}\\[2pt]
Control Systems Department, Gipsa-lab, Universit\'e de Grenoble, 11 rue des Math\'ematiques, BP 46, 38402 Saint-Martin d'H\`eres, France\\[6pt]
{\sc Fr\'ed\'eric Gouaisbaut and Sophie Tarbouriech}\\[2pt]
LAAS-CNRS, Universit\'e de Toulouse, CNRS, UPS, Toulouse, France
\\[6pt]
{\rm \bf \footnotesize This article has been published in the IMA Journal of Mathematical Control \& Information Published by Oxford University Press\\
Volume 37, Issue 2, June 2020, Pages 395–417, \href{https://doi.org/10.1093/imamci/dnz003}{https://doi.org/10.1093/imamci/dnz003}. If you find this paper useful for your research, please cite the published version.}\\[2pt]
\vspace*{6pt}}
\pagestyle{headings}
\markboth{F. Ferrante, F. Gouaisbaut, and S. Tarbouriech}{\rm\mytitle }
\maketitle


\begin{abstract}
{Stability and stabilization for linear state feedback control systems in the presence of sensor quantization are studied. As the closed-loop system is described by a discontinuous right-hand side differential equation, Krasovskii solutions (to the closed-loop system) are considered. 
Sufficient conditions in the form of matrix inequalities are proposed to characterize uniform global asymptotic stability of a compact set containing the origin. Such conditions are shown to be always feasible whenever the quantization free closed-loop system is asymptotically stable. Building on the obtained conditions, computationally affordable algorithms for the solutions to the considered problems are illustrated. The effectiveness of the proposed methodology is shown in three examples\footnote{\textcolor{blue}{This file contains a fix to a typographical error in the published paper. The error is in blue font and there is a footnote explaining it. Last update: \today.}
} }
{Quantization, Discontinuous-Control, Linear Matrix Inequalities, Krasovksii Solutions.}
\end{abstract}
\section{Introduction}
Recently technology enhancements have enabled the conception of a new generation of engineered systems integrating physical interactions and computational and communication abilities.
This new trend has been having a strong impact also in modern control systems that are nowadays built via the adoption of digital controllers and digital instrumentation \cite{murray2003future}.
When a physical system interacts with a digital one, side effects as time-delays, asynchronism, and quantization 
are unavoidable issues that can often turn into an excessive performance degradation, like the appearing of limit cycles or chaotic phenomena or even instability of the closed-loop system.
Concerning the effect of quantization in control systems, since such a phenomenon is almost pervasive in modern engineered control systems, its study has attracted researchers over the last years; see, \eg, \cite{brockett2000quantized,ceragioli2007discontinuous,Coutinho2010,Delchamps1990,fu2005sector,lib/book2003,Paden,SoFred} just to cite a few.  

In this paper, we address stability analysis and stabilization for linear state feedback control systems in the presence of sensor quantization. 
Specifically, as in \cite{cer:dep:fra/automatica2011}, we model the uniform quantizer as a discontinuous static isolated nonlinearity entering into the dynamics of the closed-loop system. At this stage, since the resulting closed-loop system is described by a discontinuous right-hand side differential equation, likewise to \cite{ceragioli2007discontinuous, krasovskiui1963stability,FerranteAutomatica2015}, we adopt the notion of Krasovskii solutions. This allows us to overcome the issues related to the existence of solutions and to exploit a large number of results available in the literature \cite{ceragioli2007discontinuous,teel2010smooth}. 
Then, via the use of the sector conditions for the uniform quantizer proposed in \cite{Ferrante2014ECC}, we propose sufficient conditions in the form of matrix inequalities to ensure uniform global asymptotic stability of a  compact set $\mathcal{A}$ surrounding the origin, (global asymptotic stability of the origin is impossible to achieve whenever the plant is not open loop asymptotically stable; see, \eg, \cite{cer:dep:fra/automatica2011,Ferrante2014ECC,SoFred}). Such conditions, although conservative, can be exploited in a convex optimization setup to dramatically decrease the size of the resulting attractor with respect to other approaches, \eg, \cite{lib/book2003}.
Next, building on the conditions issued from the stability analysis problem, by relying on the projection lemma~\cite{pipeleers2009extended}, we propose a computationally affordable algorithm to design a stabilizing controller allowing to shrink the size of the set $\mathcal{A}$. Such an algorithm is shown to be effective in several numerical examples.
\subsection*{Contribution}
The main message we want to convey in this paper is that tools 
from discontinuous right-hand side differential equations can be employed in a constructive perspective via the use of semidefinite programming. 

The contributions of this paper are as follows. Firstly, we formally prove that the sector conditions provided in 
\cite{Ferrante2014ECC} can be exploited to deal with the Krasovskii regularization of the closed-loop system. This aspect was informally argued in \cite{Ferrante2014ECC}.
A unique feature of this paper consists of showing that the set wherein the closed-loop trajectories converge is 
uniformly globally asymptotically stable. It is worthwhile to remark that the majority of the results available on quantized control systems essentially focus only on ultimate boundedness properties for the closed-loop system; see \cite[Chapter 5]{lib/book2003}, \cite{ceragioli2007discontinuous,SoFred}, just to mention a few. On the other hand, uniform global asymptotic stability is relevant since it is robustly maintained in the presence of small perturbations as, \eg, measurement noise that inevitably occur in practice; see \cite{clarke2010discontinuous}.

Secondly, we show that the constructive result we propose for stability analysis is lossless with respect to the more general results given, {\itshape e.g.}, in \cite{lib/book2003}, in the sense that if $A+BK$ is Hurwitz, then the conditions we propose are always feasible. Thirdly, building on the lossless nature of the stability analysis conditions, we propose an iterative computationally affordable algorithm for the solution to the stabilization problem, for which convergence is always ensured. Such an algorithm allows to solve the stabilization problem while considering some optimization aspects by leveraging on convex optimization. Numerical comparisons with alternative algorithms, as \cite{kocvara2005penbmi}, are offered throughout the paper.

The remainder of this paper is organized as follows. Section~\ref{sec:preliminaries} provides some background on differential inclusions. 
Section~\ref{sec:ProbStatement} is aimed at illustrating the problems we solve. Section~\ref{sec:MainResults} presents the main results of the paper. Section~\ref{sec:Numerical} is devoted to numerical issues. Section~\ref{sec:Examples} illustrates the effectiveness of the proposed methodology in three numerical examples.

{\small
{\bf Notation:} 
The symbol $\mathbb{B}$ denotes the unit closed Euclidean ball. The symbols $\Id_n$ denotes the identity matrix, while $\0$ denotes the null matrix (equivalently the null vector) of appropriate dimensions. 
For a matrix $A\in\mathbb{R}^{n\times m}$, $A\tr$, denotes the transpose of $A$, $\Tr(A)$ denotes its trace, and $\He(A)=A+A\tr$. The matrix $\Diag\{A_1,A_2,\dots,A_n\}$ is the block-diagonal matrix having $A_1,A_2,\dots,A_n$ as diagonal blocks. In symmetric matrices $\bullet$ stands for symmetric blocks. With $\Spn$, we denote the set of  the $n\times n$ symmetric positive definite matrices, while  $\Dpn$ denotes the set of  $n\times n$ diagonal positive definite matrices. Given $A\in\R^{n\times n}$, $\mathcal{R}(A)=\{\vert\Re(\lambda)\vert\colon\lambda\in\spec(A)\}$, where $\Re(\lambda)$ represents the real part of $\lambda$.
	For a vector $x\in\mathbb{R}^n$, $x\tr$ denotes its transpose, $\vert x\vert$ stands for the componentwise absolute value of $x$, $\Sign(x)$ is the componentwise sign function, with $\Sign(0)=0$, and $\lfloor x\rfloor$ is the componentwise floor function. The set $\Delta\mathbb{Z}^p$ is the set of the p-tuples of integers multiples of $\Delta$. The symbol $\langle \cdot,\cdot\rangle$ denotes the standard Euclidean inner product and $\times$ stands for the standard Cartesian product. For a set $U$, $\Int U$ denotes the interior of $U$ and $\Co U$ denotes its convex hull. The double arrows notation $F\colon \mathbb{R}^n\rightrightarrows \mathbb{R}^m$ indicates that $F$
	is a set-valued mapping with $F(x)\subset \mathbb{R}^m$. Given a point $y$ and a closed set $\mathcal{A}$, $\vert y \vert_{\mathcal{A}}$ stands for the distance of $y$ to $\mathcal{A}$. A function $\alpha\colon \mathbb{R}_{\geq 0} \rightarrow \mathbb{R}_{\geq 0}$ is said to belong to the class ${\cal K}_\infty$, also denoted $\alpha\in\mathcal{K}_\infty$,  if it is continuous, zero at zero, strictly increasing, and unbounded.
	A function $\beta\colon\mathbb{R}_{\geq 0}\times \mathbb{R}_{\geq 0}\rightarrow \mathbb{R}_{\geq 0}$ is said to belong to class ${\cal K}{\cal L}$, also denoted $\beta\in{\cal K}{\cal L}$, if it is nondecreasing in its first argument, nonincreasing in its second argument, and $\lim_{s\rightarrow 0^+}  \beta(s, t) =\lim_{t\rightarrow +\infty}  \beta(s,t) = 0$.}
\section{Background on differential inclusions}
\label{sec:preliminaries}
\subsection{Preliminary definitions}
In this paper we consider differential inclusions in the following form
\begin{equation}
\label{eq:DiffInc}
\dot{x}\in F(x)\qquad x\in\R^n\quad F\colon\R^n\rightrightarrows\R^n
\end{equation}
For such differential inclusions, we consider the following notion of solution.
\begin{definition}[\cite{aubin1984differential, goebel2012hybrid}]
	An absolutely continuous function $\phi\colon\dom\phi\rightarrow \mathbb{R}^n$ is a solution to \eqref{eq:DiffInc} if  $\dom\phi$ is an interval 
	of $\R_{\geq 0}$
	containing $0$ and $$
	\dot{\phi}(t)\in F(\phi(t))\qquad\mbox{for\,almost\,all}\quad  t\in\dom\phi
	$$
	A solution $\phi$ to \eqref{eq:DiffInc} is said to be maximal if there does not exist any other solution $\psi$ to \eqref{eq:DiffInc} such that $\dom\phi$ is a proper subset of $\dom\psi$ and $\phi(t)=\psi(t)$ for every $t\in\dom\phi$. A solution $\phi$ to \eqref{eq:DiffInc} is said to be complete if  $\sup\dom\phi=\infty$.
\end{definition}
The notions of stability  we consider in this paper for a general differential inclusion as \eqref{eq:DiffInc} and a closed set $\mathcal{A}\subset\R^n$ are given next.
\begin{definition}[Uniform global asymptotic stability]
	\label{def:Chap1:UGAS}
	Let  $\mathcal{A}\subset\R^n$ be closed. The set $\mathcal{A}$ is
	\begin{itemize}
		\item globally uniformly stable for \eqref{eq:DiffInc}, if there exists a class $\mathcal{K}_{\infty}$ function $\alpha$, such that  every solution $\phi$ to \eqref{eq:DiffInc} satisfies $\vert \phi(t)\vert_{\mathcal{A}}\leq \alpha(\vert \phi(0)\vert_{\mathcal{A}})$ for every $t\in\dom\phi$
		\item uniformly  globally attractive  for \eqref{eq:DiffInc}, if every maximal solution to \eqref{eq:DiffInc} is complete, and for every $\varepsilon>0$ and $\mu>0$ there exists $T>0$, such that for any solution $\phi$ to \eqref{eq:DiffInc} with $\vert\phi(0)\vert_{\mathcal{A}}\leq \mu$, $t\geq T$ implies  $\vert\phi(t)\vert_{\mathcal{A}}\leq \varepsilon$
		\item uniformly globally asymptotically stable (\emph{UGAS})  for \eqref{eq:DiffInc}, if it is globally uniformly stable and globally uniformly attractive
	\end{itemize}
\end{definition}
Observe that UGAS of a compact set $\mathcal{A}$ for \eqref{eq:DiffInc} ensures that every maximal solution to \eqref{eq:DiffInc} is bounded. Moreover, observe that the uniformity in the definition of UGAS pertains to the distance of the initial condition to the set $\mathcal{A}$ and not to the initial time as in \cite{Khalil}.

Next we report some technical results on differential inclusions that are used to build some of the proofs presented in this paper.
\subsection{Technical results}
Consider the following differential inclusion 
\begin{equation}
\label{eq:AppDiffInc}
\dot{x}\in F(x)\qquad x\in\R^n\qquad F\colon\R^n\rightrightarrows\R^n
\end{equation}
The first result we give essentially derives from the combined application of \cite[Proposition 7.5.]{goebel2012hybrid} and \cite[Proposition 3]{teel2010smooth}.  The derivation of such a result uses the following  definitions 
\begin{definition}[Strong forward invariance \cite{clarke2010discontinuous}]
	Let  $\mathcal{A}\subset\R^n$ be closed. The set $\mathcal{A}$ is strongly forward invariant for \eqref{eq:AppDiffInc} if every maximal solution to \eqref{eq:AppDiffInc} is complete and $\phi(0)\in\mathcal{A}$ implies $\range\phi\subset\mathcal{A}$.
\end{definition}
\begin{definition}[Basic conditions \cite{teel2010smooth}] 
	The set-valued mapping $F\colon\R^n\rightrightarrows\R^n$ is said to satisfy the basic conditions on $\R^n$ if it is outer semicontinuous on $\R^n$ and for each $x\in\R^n$, $F(x)$ is convex, non-empty, and bounded.
\end{definition} 
Now we are in position to state the mentioned result.
\begin{proposition}
	\label{prop:Chap1:UGAS}
	Consider the differential inclusion in \eqref{eq:AppDiffInc}. Let $\mathcal{A}\subset\R^n$ be compact, strongly forward invariant and uniformly globally attractive for  \eqref{eq:AppDiffInc}. Let $F$ satisfy the basic conditions on $\R^n$. Then, the set $\mathcal{A}$ is UGAS for \eqref{eq:AppDiffInc}.
\end{proposition}
\begin{proof}
	Due to the properties required for $F$ in the statement of the above result, since $\mathcal{A}$ is compact, strongly forward invariant, and uniformly globally attractive for~ \eqref{eq:AppDiffInc}, thanks to \cite[Proposition 7.5.]{goebel2012hybrid}, it follows that $\mathcal{A}$ is stable\footnote{See, \eg, \cite[Proposition 3]{teel2010smooth} for a standard definition of $\varepsilon,\delta$ stability of a compact set.} for \ \eqref{eq:AppDiffInc}. Moreover, by the virtue of \cite[Proposition 3]{teel2010smooth}, it follows that there exists a $\beta\in\mathcal{KL}$ such that for every maximal solution $\varphi$ to \eqref{eq:AppDiffInc},  one has for every $t\in\R_{\geq 0}$,
	$$
	\vert \varphi(t)\vert_{\mathcal{A}}\leq \beta(\vert\varphi(0)\vert_{\mathcal{A}},t)
	$$
	which in turn, due to  \cite[Proposition 1]{teel2010smooth}, implies that $\mathcal{A}$ is UGAS for  \eqref{eq:AppDiffInc}, and this finishes the proof.
\end{proof}
\begin{proposition}
	\label{lem:Chap1:Aux}
	Consider system \eqref{eq:AppDiffInc} and assume that $F$ satisfies the basic conditions on $\R^n$.
	Assume that there exists a continuously differentiable function $V\colon \R^n\rightarrow\R$ such that
	\begin{align}
	&V(x)>0\qquad \forall x\neq 0\\
	&\lim_{\Vert x \Vert \rightarrow\infty} V(x)=\infty
	\end{align}
	and two positive real scalars $\rho$, and $\alpha$ such that
	\begin{align}
	&\langle \nabla V(x), f\rangle \leq-\rho V (x) \qquad \forall x\in\mathcal{L}^{+}_\alpha(V), f\in F(x)
	\label{eq:Preli:LyapProof3}
	\end{align}
	where $\mathcal{L}^{+}_\alpha(V)\coloneqq \{x\in\R^n\colon V(x)\geq \alpha\}$. Then, the set $\mathcal{A}\coloneqq \R^n\setminus \Int\mathcal{L}^{+}_\alpha(V)$ is UGAS for \eqref{eq:AppDiffInc}.
\end{proposition}
The proof of the above result rests on the following lemma.
\begin{lemma}
\label{lemma:Prel:FinitTime}
Let $\mathcal{A}\subset\R^n$ be compact. If there exists a locally bounded function $\Upsilon\colon\mathbb{R}^n\rightarrow\R_{\geq 0}$ such that for each $\mu>0$ every maximal solution $\varphi$ to \eqref{eq:AppDiffInc} with $\varphi(0)\in\mathcal{A}+\mu\mathbb{B}$ is complete and $t\geq\Upsilon(x_0)$ implies $\varphi(t)\in\mathcal{A}$. Then, $\mathcal{A}$ is globally uniformly attractive for  \eqref{eq:AppDiffInc}. \QEDB
\end{lemma}
\begin{proof}
	The proof is straightforward. Pick $\mu>0$ and define
	$$
	\xi=\underset{x\in \mathcal{A}+\mu\mathbb{B}}\sup\Upsilon(x)
	$$
Then, since $\Upsilon$ is locally bounded and $\mathcal{A}$ is compact $\xi$ is finite. To conclude, notice that for each maximal solution $\varphi$ to \eqref{eq:AppDiffInc} with $\varphi(0)\in\mathcal{A}+\mu\mathbb{B}$, one has that $t\geq \xi$ implies $\varphi(t)\in\mathcal{A}$ and this concludes the proof.
\end{proof}

Now we are in position to show the proof of Proposition~\ref{lem:Chap1:Aux}.

\begin{proof}[Proof of Proposition~\ref{lem:Chap1:Aux}]
	First observe that since $V$ is radially unbounded, $\mathcal{A}$ is compact.
	To prove that the set $\mathcal{A}$ is UGAS, we firstly show that $\mathcal{A}$ is strongly forward invariant for  \eqref{eq:AppDiffInc} and that each maximal solution to \eqref{eq:AppDiffInc} is complete.
	
	Concerning strongly forward invariance, since $\mathcal{A}$ is compact, thanks to the properties required for $F$, from \cite[Proposition 6.10.]{goebel2012hybrid}, it suffices to show that each maximal solution from $\mathcal{A}$ cannot leave such a set, \ie, completeness of such solutions automatically holds. On the other hand, this follows from \eqref{eq:Preli:LyapProof3}. 
In particular, the satisfaction of \eqref{eq:Preli:LyapProof3} ensures that every maximal solution $\varphi$ to  \eqref{eq:AppDiffInc} with $\varphi(0)\notin\mathcal{A}$ cannot leave the sublevel set $\mathcal{L}^{-}_{V(\varphi(0))}(V)\coloneqq\{x\in\R^n\colon V(x)\leq V(\varphi(0))\}$. 
	Hence, since sublevel sets of $V$ are compact, it follows that every maximal solution to \eqref{eq:AppDiffInc} is bounded and then complete. Hence strong forward invariance of $\mathcal{A}$ and completeness of maximal solutions to   \eqref{eq:AppDiffInc} are proven.
	
Bearing in mind completeness of maximal solutions to \eqref{eq:AppDiffInc} and strong forward invariance of $\mathcal{A}$, we conclude the proof of the above result by showing that maximal solutions to \eqref{eq:AppDiffInc} converge in finite time into $\mathcal{A}$.
Pick any maximal solution $\varphi$ to \eqref{eq:AppDiffInc}, with $\varphi(0)\in \Int\Theta$.
Let $\mathcal{T}=\{t\in\R_{\geq 0}\colon \varphi(t)\in\mathcal{A} \}$, since $\mathcal{A}$ is strongly forward invariant, either $\mathcal{T}=\emptyset$ or $\sup\mathcal{T}=+\infty$. In other words, if $\varphi$ eventually enters $\mathcal{A}$, then by strong forward invariance, it cannot leave such a set.
	
By contradiction, let us suppose that  $\mathcal{T}=\emptyset$, then for every $t\in\R_{\geq 0}$, $\varphi(t)\notin\mathcal{A}$. 
Therefore, still from \eqref{eq:Preli:LyapProof3}, it follows that
	\begin{equation}
	\label{eq:Chap1:Bound}
	V(\varphi(t))\leq e^{-\rho t} V(\varphi(0))\qquad \forall t\in\R_{\geq 0}.
	\end{equation}
	Pick $t\geq\frac{1}{\rho}\ln\left(V(\varphi(0))\frac{1}{\alpha}\right)$, from \eqref{eq:Chap1:Bound} one gets
	$
	V(\varphi(t))\leq \alpha 
	$
	that is $\varphi(t)\in\mathcal{A}$, but this contradicts the fact that $\mathcal{T}=\emptyset$. Now, for every $w\in\R^n$, define
	$$
	\Upsilon(w)\coloneqq\begin{cases}0& w\in\mathcal{A}\\
	\frac{1}{\rho}\ln\left(V(w)\frac{1}{\alpha}\right)&w\notin\mathcal{A}
	\end{cases}
	$$
	Notice that $\Upsilon$ is continuous on $\R^n$, hence locally bounded.  Moreover, for every maximal solution $\phi$ to  \eqref{eq:AppDiffInc}, $t\geq\Upsilon(\phi(0))$ implies that $\phi(t)\in\mathcal{A}$. Then, since every maximal solution to  \eqref{eq:AppDiffInc} is complete, from Lemma~\ref{lemma:Prel:FinitTime} it follows that $\mathcal{A}$ is globally uniformly attractive for \eqref{eq:AppDiffInc}. 
Now, $\mathcal{A}$ is compact, strongly forward invariant, and globally uniformly attractive for \eqref{eq:AppDiffInc}, from Proposition~\ref{prop:Chap1:UGAS} it follows that $\mathcal{A}$ is UGAS for  \eqref{eq:AppDiffInc} and this finishes the proof.
\end{proof}
\section{Problem statement}
\label{sec:ProbStatement}
\subsection{System Description}
Consider the following continuous-time linear system
\begin{equation}
\label{eq:Chap1:QuantizedState}
\dot{x}=Ax+Bu
\end{equation}
where $x\in\R^n$ and $u\in\R^m$ are, respectively, the state and the input of the system while $A\in\R^{n\times n}$ and $B\in\R^{n\times m}$. We assume that the state of the plant \eqref{eq:Chap1:QuantizedState} is measured in a quantized fashion. In particular, let us denote by $x_m$ the measurement of the plant state. Then, we assume that $x_m=\q(x)$, where $\q$ is the \emph{quantizer}, that is a map $q\colon\R^n\rightarrow\mathcal{Q}$, where $\mathcal{Q}\subset\R^{n}$ is a countable set.   
In particular, let $\delta_i$, for $i=1,2,\dots, n$, be some positive given scalars, in this work, we consider the following ``uniform'' quantizer $q\colon\R^n\rightarrow \delta_1\mathbb{Z}\times\delta_2\mathbb{Z}\times\dots\delta_n\mathbb{Z}$, for which 
\begin{equation}
\begin{aligned}
&\q_i(x)\coloneqq \delta_i\Sign(x_i)\left\lfloor \frac{\vert x_i \vert}{\delta_i}\right\rfloor\qquad i=1,2,\dots,n
\end{aligned}
\label{eq:Chap1:Uniformquantizer}
\end{equation}
the scalars $\delta_i$, for $i=1,2,\dots, n$, characterize the quantization error bounds, \ie, for every $x\in\R^n$ and $i=1,2,\dots, n$, $\vert\q(x_{(i)})-x_{(i)}\vert\leq\delta_i$. It is worthwhile to remark that for each $i=1,2,\dots, n$, $\q_i([-\delta_i,  \delta_i])=\{0\}$. In the sequel, we define $\Delta=(\delta_1,\delta_2,\dots, \delta_n)$.

Assume now that the measurement of the state $x_m$ is utilized in a static feedback control scheme to stabilize the plant \eqref{eq:Chap1:QuantizedState}. More specifically, let $K\in\R^{m\times n}$ and let us consider the following control law $u=Kx_m$; see \figurename~\ref{fig:scheme_cl}. The objective of this paper is to study the stability of the resulting closed-loop system, which can be modeled as follows
\begin{equation}
\label{eq:Chap1:QuantizedState_CL}
\dot{x}=Ax+BK\q(x)
\end{equation}
	
\begin{figure}[!ht]
		\centering
		\psfrag{xm}[][][1]{$x_m$}
		\psfrag{x}[][][1]{$x$}
		\psfrag{q}[][][1]{$\q(\cdot)$}
		\psfrag{K}[][][1]{$K$}
		\psfrag{P}[][][1.2]{$\dot{x}=Ax+Bu$}
		\psfrag{u}[][][1.1]{$u$}
		\includegraphics[scale=0.8]{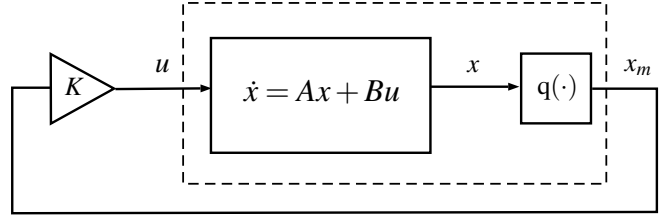}
		\caption{The closed-loop system.}
		\label{fig:scheme_cl}
	\end{figure}
As extensively illustrated in the literature of quantized control systems; see, \eg, \cite{brockett2000quantized,lib/book2003, SoFred,FerranteAutomatica2015}, due to the dead-zone effect induced by the quantizer $\q$, if the matrix $A$ is not Hurwitz, then the origin of \eqref{eq:Chap1:QuantizedState_CL} cannot be rendered asymptotically stable via any choice of the gain $K$. 
Nevertheless, as shown in \cite{lib/book2003,SoFred}, it turns out that if the matrix $A+BK$ is Hurwitz, then the  closed-loop trajectories converge in finite time into a compact set $\mathcal{A}$ containing the origin. 
Building on the sector conditions for the uniform quantizer presented in \cite{FerranteAutomatica2015,Ferrante2014ECC}, one first goal of this paper consists of deriving computationally tractable conditions aimed at providing a systematic characterization of the set $\mathcal{A}$ starting from the data of the closed-loop system (\emph{stability analysis problem}). Moreover, we show that the set $\mathcal{A}$ is globally uniformly asymptotically stable for the closed-loop system. 
Such a property is relevant as, thanks to the results in \cite{goebel2012hybrid,teel2010smooth}, it is semiglobally robustly preserved in the presence of small perturbations like measurement noise.
As a second problem, we address the design problem of a stabilizing gain $K$ for \eqref{eq:Chap1:QuantizedState_CL} aimed at reducing the effect of the quantization on the closed-loop system by reducing the ``size'' of the set $\mathcal{A}$ (\emph{stabilization problem}). 

Define the function, 
\begin{equation}
\label{eq:Chap1:Psi}
\begin{aligned}
\Psi\colon&\mathbb{R}^n\rightarrow\R^n\\
&x\mapsto \q(x)-x
\end{aligned}
\end{equation}
which maps $x$ into the quantization error introduced by the quantizer in \eqref{eq:Chap1:Uniformquantizer}. Then, the closed-loop system can be rewritten as follows.
\begin{equation}
\label{eq:Chap1:QuantizedStateCL}
\dot{x}=(A+BK)x+BK\Psi(x)
\end{equation}
Observe that being the function $\Psi$ discontinuous, \eqref{eq:Chap1:QuantizedStateCL} is a discontinuous right-hand side differential equation. The notion of solution to be considered in studying such a class of systems needs to be carefully chosen; see \cite{ceragioli2007discontinuous}. One notion usually adopted in the literature of discontinuous dynamical systems is the notion of Carath\'eodory solution, which is given as follows.
\begin{definition}
	Let $\mathbb{I}\subset\R_{\geq 0}$ be an interval containing $0$. An absolutely continuous function $\phi\colon \mathbb{I}\rightarrow \R^n$ is a  Carath\'eodory solution to \eqref{eq:Chap1:QuantizedStateCL} if for almost every $t\in \mathbb{I}$, one has 
	$$\dot{\phi}(t)=(A+BK)\phi(t)+BK\Psi(\phi(t))$$
\end{definition}
The above definition does not insist on the differentiability of the solution and on the fact that \eqref{eq:Chap1:QuantizedStateCL} needs to be satisfied on the whole domain of the solution. 
However, in some cases, the notion of solution due to Carath\'eodory is not weak enough to guarantee the existence of solutions to the considered class of systems. To understand the relevance of this issue, let us consider the following example.
\begin{example}
	\label{ex:Example0}
	Consider the static state feedback control system with quantized sensor from \cite{frid:dam/auto09} that is defined by the following data:
	$$
	A=\begin{bmatrix}
	0 &1\\
	0.5& 0.5
	\end{bmatrix}, B=\begin{bmatrix}
	1\\1
	\end{bmatrix}, \delta_1=\delta_2=1, K=\begin{bmatrix}
	-0.3491 &-0.7022
	\end{bmatrix}
	$$
	We show that there exist no Carath\'eodory solutions to the closed-loop system from each point belonging to the set $\mathcal{S}=\{x\in\mathbb{R}^2\colon x_1=x_2=k, k=\pm1,\pm 2,\dots,\pm 20\}$. Pick any $k=1,2,\dots, 20$. 
	Define $U_k(x)=\frac{1}{2}\left((x_1-k)^2+(x_2-k)^2\right)$, $\mathcal{Q}=\{x\in\mathbb{R}^2\colon x_1=x_2\}$ and observe that there exists $\varepsilon_k>0$ small enough such that
	\begin{equation}
	\label{eq:Uk}
	\langle \nabla U_k(x),Ax+BK\q(x)\rangle\leq 0 \qquad\forall x\in\left(k\1+\varepsilon_k\mathbb{B}\right)\cap\mathcal{Q}
	\end{equation}
	Moreover, observe that solutions from $\mathcal{Q}$ cannot leave such a set. To prove this claim, it suffices to notice that $W(x)=\frac{1}{2}(x_1-x_2)^2=\vert x\vert_\mathcal{Q}^2$ is such that $\langle\nabla W(x),Ax+BK\q(x)\rangle=-\vert x\vert_\mathcal{Q}^2$ for every $x\in\R^2$. 
	
	Now, assume that there exists a Carath\'eodory solution $t\mapsto\phi(t)=(\phi_1(t),\phi_2(t))$ to the closed-loop system, such that  $\phi(0)\in\mathcal{S}$. Since $\mathcal{S}\subset\mathcal{Q}$ and solutions from $\mathcal{Q}$ cannot leave such a set,  it follows that $\range\phi\subset\mathcal{Q}$. 
	Now, observe that since $\phi(t)$ is right continuous at $t=0$, and by assumption $\phi(0)=k\1$, there exists 
	$T^\star\in\dom\phi$ small enough such that $\range\phi_{T^\star} \subset\left(k\1+\varepsilon_k\mathbb{B}\right)\cap \mathcal{Q}$, where  $\dom\phi_{T^\star}=[0,T^{\star}]$ and for each $t\in\dom\phi_{T^\star}$, $\phi_{T^\star}(t)=\phi(t)$. 
	Therefore, since $U_k$ is continuously differentiable on $\R^2$ and $\phi$ is absolutely continuous, from \eqref{eq:Uk} it follows that 
	\begin{equation}
	\label{eq:IntUk}
	U_k(\phi(t))=\int_0^{t}\frac{dU(\phi(s))}{ds}ds=\int_0^{t} \langle\nabla U(\phi(s)), A\phi(s)+BK\q(\phi(s))\rangle ds\leq 0\qquad \forall t\in[0,T^{\star}]
	\end{equation}
	Where the above integral has to be considered as a Lebesgue integral.
	Therefore, being $U_k$ positive definite with respect to\footnote{Given a function $V\colon\R^n\rightarrow\R$ and a set $\mathcal{A}\subset\R^n$, we say that $V$ is positive definite with respect to $\mathcal{A}$, if $V(x)=0$ for each $x\in\mathcal{A}$ and for each $x\notin\mathcal{A}$, $V(x)>0$.} $\{(k,k)\}$, it follows that
	$U_k(\phi(t))=0$ for all $t\in[0,T^{\star}]$, \ie, $\phi(t)=k\1$ for all $t\in[0,T^{\star}]$.
	But this contradicts the fact that $\phi$ is a Carath\'eodory solution to the closed-loop system, since 
	$A\phi(t)+BK\q(\phi(t))\neq 0$ for all $t\in[0,T^{\star}]$.
\end{example}
To overcome this shortcoming, in this paper we consider Krasovskii solutions to \eqref{eq:Chap1:QuantizedStateCL}. In particular, Krasovskii solutions are defined as follows for a general differential equation of the form 
\begin{equation}
\label{eq:ODE}
\dot{x}=X(x)\qquad x\in\R^n,  X\colon\R^n\rightarrow\R^n
\end{equation}
\begin{definition}[Krasovskii solution \cite{Hajek:1979aa}]
	\label{def:Chap1:Krasovskii}
	For each $x\in\R^n$, let us define  the following set-valued mapping
	$$
	\mathcal{K}[X](x)\coloneqq\bigcap_{\delta>0}\ch X(x+\delta\mathbb{B})
	$$
	where $\mathbb{B}$ is the closed unitary ball in $\R^n$.
	
	Let $\mathbb{I}\subset\R_{\geq 0}$ be an interval containing $0$. An absolutely continuous function $\phi\colon\mathbb{I}\rightarrow \mathbb{R}^n$ is a Krasovskii solution to \eqref{eq:ODE} if 
	$$
	\dot{\phi}(t)\in \mathcal{K}[X](\phi(t))\qquad \text{for almost all }\,\,t\in\mathbb{I}
	$$
\end{definition}

\begin{remark}
Looking again at Example~\ref{ex:Example0}, it turns out that Krasovskii solutions from $\mathcal{S}$ do exist and are unique. In particular, for each $x_0\in\mathcal{S}$, $\phi(t)=x_0$ with $\dom\phi=\R_{\geq 0}$  is a (maximal) Krasovskii solution. In other words, the set $\mathcal{S}$ is made of Krasovskii equilibria for the considered system.
\end{remark}
Three main reasons encourage to focus on Krasovskii solutions in control problems. The first one is that Krasovskii solutions exist under very mild requirements, in particular if $X$ is locally bounded, then for every $x_0\in\mathbb{R}^n$, there exists at least a Krasovskii solution $\phi$ to \eqref{eq:ODE}, such that $\phi(0)=x_0$; see \cite{Hajek:1979aa}.
The second one is that, whenever they exist, Carath\'eodory solutions are Krasovskii solutions, then any conclusion drawn on Krasovskii solutions also holds for Carath\'eodory solutions. 
The third one is that, as shown in \cite[Corollary 5.6.]{Hajek:1979aa}, (and also more recently in \cite[Theorem 4.3.]{goebel2012hybrid}), Krasovskii solutions have connections with the perturbed dynamics of \eqref{eq:ODE}. Thus, results established on Krasovskii solutions are (practically-semiglobally) preserved in the presence of small perturbations.

For each $x\in\R^n$, define
$$Z=(A+BK)x+BK\Psi(x)$$ 
From now on, we shall denote as closed-loop system the following differential inclusion
\begin{subequations}
	\label{eq:Chap1:Kraso}
	\begin{equation}
	\dot{x}\in\mathcal{K}[Z](x)
	\end{equation}
	whose solutions are the Krasovskii solutions to \eqref{eq:Chap1:QuantizedStateCL}. In particular, by following similar arguments to those in \cite{Paden:1987aa}, it can be readily shown that for each $x\in\R^n$
	\begin{equation}
	\mathcal{K}[Z](x)=(A+BK)x+BK\mathcal{K}[\Psi](x).
	\end{equation}
\end{subequations}
Now we are in position to state the problems we solve in this paper.
\begin{problem}(Stability analysis)
	\label{prob:Chap1:StabilityQuantizedState}
	Let $A\in\R^{n\times n}, B\in\R^{n\times m}$, and $K\in\R^{m\times n}$ be given. Determine a compact set $\mathcal{A}\subset\mathbb{R}^n$ containing the origin, such that $\mathcal{A}$ is UGAS for system \eqref{eq:Chap1:Kraso}.
\end{problem}
\begin{problem}(Stabilization)
	\label{prob:Chap1:StabilizationQuantizedState}
	Let $A\in\R^{n\times n}$ and $B\in\R^{n\times m}$ be given. Determine a gain $K\in\R^{m \times n}$ and a compact set $\mathcal{A}\subset\mathbb{R}^n$ containing the origin, such that $\mathcal{A}$ is UGAS for system \eqref{eq:Chap1:Kraso}.
\end{problem}
\begin{remark}
	Following the lines of existing works on quantized systems (\cite{lib/book2003,SoFred}), a  seemingly similar stability notion could be used in the framework addressed by our paper, that is the notion of global ultimate boundedness with respect to a compact set $\mathcal{A}$; see \cite{Khalil}. 
However,  it is worthwhile to remark that global ultimate boundedness  {\itshape a priori} does not guarantee Lyapunov stability of the considered compact set $\mathcal{A}$. Thus, trajectories starting near $\mathcal{A}$ could dramatically deviate from it. This is obviously impossible whenever one focuses on uniform asymptotic stability. Moreover, as shown in \cite{goebel2012hybrid, teel2010smooth}, uniform global asymptotic stability of compact sets is robust with respect to small perturbations.  
\end{remark}
\section{Main results}
\label{sec:MainResults}
First, let us consider the following preliminary results inspired by \cite{Ferrante2014ECC}.
\begin{lemma}
\label{lem:sec1}
Let $z\in\mathbb{R}^n$ and $S_1\in\Dpn$. Then, the following relation holds
\begin{equation}
\label{eq:Chap1:Sector1.0}
\Psi(z)\tr S_1\Psi(z)-\Delta\tr S_1\Delta\leq 0
\end{equation}
\end{lemma}
\begin{proof}
The proof is straightforward. Pick any $z\in\R^n$, then, since for each $i=1,2,\dots,n$, $\vert \Psi_i(z)\vert\in[-\delta_i, \delta_i]$, one has that for any $S_1\in\Dpn$ and each $i=1,2,\dots,n$, $S_{1_{i,i}}\Psi_i(z)^2\leq S_{1_{i,i}}\delta_i^2$. Summing up on $i=1,2,\dots, n$ one gets \eqref{eq:Chap1:Sector1.0}.
\end{proof}

\begin{lemma}
\label{lem:Chap1:SectorKrasovskii}
Let $z\in\mathbb{R}^n$, $v\in\mathcal{K}[\Psi](z)$, and $S_1, S_2\in\Dpn$. Then, the following relations hold:
\begin{align}
\label{eq:Chap1:Sector1}
&v\tr S_1v-\Delta\tr S_1\Delta\leq 0\\ 
\label{eq:Chap1:Sector2}
&v\tr \,S_2\,(v+z)\leq 0
\end{align}
\end{lemma}
\begin{proof}
For each $z\in\mathbb{R}^n$, let us define the set 
$$\mathcal{L}(z)=\{\lim\Psi(z_k)\vert z_k\rightarrow z\}\subset \mathbb{R}^n$$
where $z_k$ is any sequence converging to\footnote{This notation is inherited by the seminal work of Paden and Sastry \cite{Paden:1987aa} presenting calculation rules for the Filippov  regularization.} $z$. 
Since $\Psi$ is locally bounded, from \cite{bacciotti2004stabilization} it turns out that for every  $z\in\mathbb{R}^n$,  $\mathcal{K}[\Psi](z)=\Co \mathcal{L}(z)$. Now, let us define the following closed set
$$\mathcal{V}_1=\{v\in\mathbb{R}^n\colon v\tr S_1v-\Delta\tr S_1\Delta\leq 0\}\subset\mathbb{R}^n$$
which, due to $S_1$ positive definite, is also convex \footnote{Positive definiteness of $S_1$ implies that the function $v\mapsto v\tr S_1v-\Delta\tr S_1\Delta$ is convex. Thus, convexity of $\mathcal{V}_1$ follows from the fact that  sublevel sets of convex functions are convex sets; see, \eg, \cite{boyd2004convex}.}. We want to show that, for  $z\in\mathbb{R}^n$
\begin{equation}
\Co\mathcal{L}(z)\subset \mathcal{V}_1
\label{eq:IncUno}
\end{equation}
To this end, pick $l\in\mathcal{L}(z)$, then, by definition, there exists a sequence $z_k\rightarrow z$ such that $l=\lim\Psi(z_k)$. On the other hand, from Lemma~\ref{lem:sec1}, it turns out that, for every $k\in\mathbb{N}$ and for every diagonal positive definite matrix $S_1$, one has $\Psi(z_k)\tr S_1\Psi(z_k)-\Delta\tr S_1\Delta\leq 0$, which, by taking the limit over $k$ yields $l\tr S_1l-\Delta\tr S_1\Delta\leq 0$, that is $l\in\mathcal{V}_1$, and then 
$$\mathcal{L}(z)\subset\mathcal{V}_1$$  
Thus, since $\mathcal{V}_1$ is convex,  taking the convex-hull of both sides of the latter relation  establishes \eqref{eq:IncUno}, which in turn gives \eqref{eq:Chap1:Sector1}.
To show \eqref{eq:Chap1:Sector2}, we pursue a similar approach. 
Specifically, for any $z\in\mathbb{R}^n$, define the closed set $$\mathcal{V}_2(z)=\{v\in\mathbb{R}^n\colon v\tr S_2(v+z)\leq 0\}\subset\mathbb{R}^n$$
which is convex due to $S_2$ positive definite. We want to show that $\Co\mathcal{L}(z)\subset\mathcal{V}_2(z)$. To this end, pick any $l\in\mathcal{L}(z)$, then there exists a sequence $z_k\rightarrow z$, such that $l=\lim\Psi(z_k)$. According to \cite{Ferrante2014ECC}, for every $k\in\mathbb{N}$, one has $\Psi(z_k)\tr S_2(\Psi(z_k)+z_k)\leq 0$, then by taking the limit over $k$, one gets $l\tr S_2(l+z)\leq 0$, that is $l\in\mathcal{V}_2(z)$, and then 
$$\mathcal{L}(z)\subset\mathcal{V}_2(z)$$  
Thus, by taking the convex hull of both sides, being $\mathcal{V}_2(z)$  convex, yields $\Co\mathcal{L}(z)\subset\mathcal{V}_2(z)$ and this finishes the proof.
\end{proof}

Next we give sufficient conditions for the solution, respectively, to Problem \ref{prob:Chap1:StabilityQuantizedState} and Problem \ref{prob:Chap1:StabilizationQuantizedState}.
\subsection{Stability analysis}
The next result, whose derivation is based on the technical results on differential inclusions reported in Section~\ref{sec:preliminaries}, gives a sufficient condition to solve Problem \ref{prob:Chap1:StabilityQuantizedState}. 
\begin{proposition}
	\label{prop:Chap1:StabilityLMIState}
	If there exist $P\in \Spn$, $S_1,S_2\in\Dpn$, and a positive scalar $\tau$ such that
	\begin{align}
	\label{eq:Chap1:EqProp1QuantizedState}
	&\begin{bmatrix}
	\He(P(A+BK))+\tau P&PBK-S_2\\
	\bullet&-S_1-2S_2
	\end{bmatrix}<\0\\
	\label{eq:Chap1:EqProp2QuantizedState}
	&\Delta\tr S_1\Delta-\tau\leq 0
	\end{align}
	then,
	\begin{equation}
	\mathcal{A}=\mathcal{E}(P)
	\label{eq:Chap1:AQuantizedState}
	\end{equation}
	is a solution to Problem~\ref{prob:Chap1:StabilityQuantizedState}.
\end{proposition}
\begin{proof}
	For every $x\in\R^n$, consider the following quadratic function $V(x)= x\tr Px$  and assume that 
	\eqref {eq:Chap1:EqProp1QuantizedState} and \eqref{eq:Chap1:EqProp2QuantizedState} hold. We want to prove that  there exists a positive real scalar $\beta$ such that 
	\begin{equation}
	\label{eq:Chap1:Vp0State}
	\langle \nabla V(x),(A+BK)x+BKv\rangle \leq -\beta V(x)\qquad \forall x\in\mathbb{R}^n\setminus \Int\mathcal{A},\,v\in\mathcal{K}[\Psi](x)
\end{equation} 
Then, UGAS of the set $\mathcal{A}$ in \eqref{eq:Chap1:AQuantizedState} for \eqref{eq:Chap1:Kraso} follows directly from Proposition~\ref{lem:Chap1:Aux}.
Now observe that for each $x\in\mathbb{R}^n\setminus \Int\mathcal{A},v\in\mathcal{K}[\Psi](x)$, and any positive scalar $\tau$
$$
\langle\nabla V(x),(A+BK)x+BKv\rangle \leq \langle\nabla V(x),(A+BK)x+BKv\rangle -\tau(1-x\tr Px)
$$
Hence, to show \eqref{eq:Chap1:Vp0State} it suffices to show that for every $x\in\mathbb{R}^n$
\begin{equation}
\label{eq:Chap1:VpState}
\langle \nabla V(x),(A+BK)x+BKv\rangle -\tau(1-x\tr Px)\leq -\beta V(x) \qquad \forall v\in\mathcal{K}[\Psi](x)
\end{equation}
By using similar arguments, thanks to Lemma~\ref{lem:Chap1:SectorKrasovskii}, it follows that \eqref{eq:Chap1:VpState} holds if
\begin{equation*}
\begin{split}
&\langle\nabla V(x),(A+BK)x+BKv\rangle-\tau(1-x\tr Px)-v\tr S_1v\\
&+\Delta\tr S_1\Delta-2v\tr S_2(v+x)\leq -\beta  V(x)\qquad \forall x\in\mathbb{R}^n, v\in\mathbb{R}^n
\end{split}
\end{equation*}
	By straightforward calculations, the left-hand side of the above relation can be rewritten as follows
	\begin{equation}
	{\begin{bmatrix}
		x\\
		v
		\end{bmatrix}}^T \begin{bmatrix}
	\He(P(A+BK))+\tau P&PBK-S_2\\
	\bullet&-S_1-2S_2
	\end{bmatrix}\begin{bmatrix}
	x\\
	v
	\end{bmatrix}+\Delta\tr S_1\Delta-\tau
	\end{equation} 
	Thus in view of \eqref{eq:Chap1:EqProp1QuantizedState} and \eqref{eq:Chap1:EqProp2QuantizedState}, it follows that there exists a small enough positive scalar $\gamma$ such that for every $x\in\mathbb{R}^n\setminus\Int\mathcal{A}, v\in\mathcal{K}[\Psi](x)$, one has $\langle \nabla V(x),(A+BK)x+BKv\rangle \leq -\gamma x\tr x$. Then, since for every $x\in\R^n$, $V(x)\leq\lambda_{\max}(P)x\tr x$, by setting $\beta=\frac{\gamma}{\lambda_{\max}(P)}$ gives \eqref{eq:Chap1:VpState} and this finishes the proof.
\end{proof}

\begin{remark}
In the proof of Proposition~\ref{prop:Chap1:StabilityLMIState}, UGAS of the set $\mathcal{A}=\mathcal{E}(P)$ is proven by showing finite time attractivity and strong forward invariance of such a set. This is possible due to the regularity of the right-hand side of \eqref{eq:Chap1:Kraso} that allows to rely on the results in 
\cite{goebel2012hybrid, teel2010smooth}. In other words, in this specific case Lyapunov stability of the set $\mathcal{A}$ directly comes from finite time convergence and strong forward invariance. However, this is not always the case.
\end{remark}

\begin{remark}
In deriving the above result, we selected the set $\mathcal{A}$ to be the $1$-sublevel set of the function $V$. This is not restrictive. Indeed, if one selects the set $\mathcal{A}$ to be any other nonempty sublevel set of the function $V$, then the resulting conditions, up to a rescaling and an invertible change of variables, turn out to be equal to \eqref{eq:Chap1:EqProp1QuantizedState} and \eqref{eq:Chap1:EqProp2QuantizedState}. This is a common custom adopted in the literature; see, \eg, \cite{tarbouriech2011}.
\end{remark}
Proposition~\ref{prop:Chap1:StabilityLMIState} provides a sufficient condition to solve Problem~\ref{prob:Chap1:StabilityQuantizedState}.  A necessary condition to ensure the feasibility of \eqref{eq:Chap1:EqProp1QuantizedState} is that the matrix $A+BK$ is Hurwitz. The result given next shows that such a condition is also sufficient for the simultaneous satisfaction of \eqref{eq:Chap1:EqProp1QuantizedState} and \eqref{eq:Chap1:EqProp2QuantizedState}.

\begin{proposition}
	\label{prop:Chap1:ExistenceStability}
	Let $K\in\mathbb{R}^{m\times n}$ such that $A+BK$ is Hurwitz. Then, there exists $(\tau,P,S_1,S_2)\in\R_{>0}\times\Spn\times \Dpn\times \Dpn$ satisfying \eqref{eq:Chap1:EqProp1QuantizedState} and \eqref{eq:Chap1:EqProp2QuantizedState}.
\end{proposition}
\begin{proof}
	Assume there exists $(\overline{\tau},\overline{P},\overline{S}_1)\in\R_{>0}\times\Spn\times \Dpn$ such that
	\begin{align}
	\label{eq:Chap1:Existence0C1}
	&\begin{bmatrix}
	\He(\overline{P}(A+BK))+\overline{\tau} \overline{P}K&\overline{P}B\\
	\bullet&-S_1
	\end{bmatrix}<\0\\
	&\Delta\tr \overline{S}_1\Delta-\overline{\tau}\leq 0
	\label{eq:Chap1:Existence0C2}
	\end{align}
	For every diagonal $S_2\in\R^{p\times p}$, define 
	$$
	\mathcal{M}(S_2)\coloneqq \begin{bmatrix}
	\He(\overline{P}(A+BK))+\overline{\tau}\overline{P}&\overline{P}BK-S_2\\
	\bullet&-\overline{S}_1-2S_2
	\end{bmatrix}.
	$$
	From \eqref{eq:Chap1:Existence0C1} it follows that $\mathcal{M}(\0)<\0$. Moreover, since $\mathcal{M}(S_2)$ depends continuously on the entries of $S_2$, there exists a small enough positive scalar $\delta$, such that for every $S_2\in\delta\Dpm$ with $S_2\leq \delta\Id$  yields\footnote{This fact can be justified by noticing that the set $H\coloneqq\{v\in\R^m\colon\mathcal{M}(\Diag\{v_1,v_2,\dots, v_m\})<\0\}$ is open. Then, since $0\in H$, there exists a positive scalar $\varepsilon$ such that $\varepsilon\mathbb{B}\subset H$. Thus, by picking $\delta=\frac{1}{\sqrt{m}}\varepsilon$ yields the result.} $\mathcal{M}(S_2)<\0$, hence proving \eqref{eq:Chap1:EqProp2QuantizedState}. 
	
	To conclude the proof, it suffices to show that whenever $A+BK$ is Hurwitz there exists $(\overline{\tau},\overline{P},\overline{S_1})\in\R_{>0}\times\Spn\times \Dpn$  such that \eqref{eq:Chap1:Existence0C1} and \eqref{eq:Chap1:Existence0C2} hold.
	To this end, define $A_{cl}=A+BK$ and let $\mathcal{R}(A_{cl})\coloneqq\{\vert\Re(\lambda)\vert\colon\lambda\in\spec(A_{cl})\}$; notice that since $A_{cl}$ is Hurwitz, then $\mathcal{R}(A_{cl})\subset\R_{>0}$. 
	Pick $\bar{\tau}\in(0,2\min \mathcal{R}(A_{cl}))$, and define, $\widetilde{A}_{cl}=A_{cl}+\frac{\bar{\tau}}{2}\Id$. Observe that, according to the selection considered for $\bar{\tau}$, $\widetilde{A}_{cl}$ is Hurwitz.
Select $\overline{S}_1\in\Dpn$, such that $\Delta\tr\overline{S}_1\Delta-\overline{\tau}\leq 0$.  By following these choices, the right-hand side of \eqref{eq:Chap1:Existence0C1} reads
	\begin{equation}
	\label{eq:Chap1:Existence1}
	\begin{bmatrix}
	\He(\widetilde{A}_{cl}\tr P)&PBK\\
	\bullet&-\overline{S}_1
	\end{bmatrix}.
	\end{equation}
	For any $	\overline{Q}_1\in\Spn$, pick the solution $\overline{W}\in\Spn$ to the following matrix equality
	$
	\He(\widetilde{A}_{cl}\overline{W})=-BK\overline{S}_1^{-1}K\tr B\tr -\overline{Q}_1.
	$
	Note that such a solution always exists since $\widetilde{A}_{cl}$ is Hurwitz, and $\overline{S}_1\in\Dpn$.  Now, set in \eqref{eq:Chap1:Existence1}, $P=\overline{W}^{-1}$, then \eqref{eq:Chap1:Existence1} becomes
	\begin{equation}
	\label{eq:Chap1:Existence2}
	\begin{bmatrix}
	\He(\widetilde{A}_{cl}\tr \overline{W}^{-1})&\overline{W}^{-1}BK\\
	\bullet&-\overline{S}_1
	\end{bmatrix}.
	\end{equation}
	We want to show that the latter matrix is negative definite.  By pre-and-post-multiplying \eqref{eq:Chap1:Existence2} by $\Diag(\overline{W}, \Id)$, it turns out that \eqref{eq:Chap1:Existence2} is negative definite if and only if
	\begin{equation}
	\label{eq:Chap1:Existence3}
	\begin{bmatrix}
	\He(\widetilde{A}_{cl}\overline{W})&BK\\
	\bullet&-\overline{S}_1
	\end{bmatrix}<\0
	\end{equation}
	and the latter, due to the selection done for $\overline{W}$ turns into
	\begin{equation}
	\label{eq:Chap1:Existence4}
	\begin{bmatrix}
	-BK\overline{S}_1^{-1}K\tr B\tr -\overline{Q}_1&BK\\
	\bullet&-\overline{S}_1
	\end{bmatrix}<\0
	\end{equation}
	Moreover, by Schur complement, as $\overline{S}_1$ is positive definite, \eqref{eq:Chap1:Existence4} is negative definite if and only if
	$$
	-BK\overline{S}_1^{-1}K\tr B\tr -\overline{Q}_1+BK\overline{S}_1^{-1}K\tr B\tr=-\overline{Q}_1<\0
	$$
which is obviously satisfied due to $\overline{Q}_1\in\Spn$. Then, $(\overline{\tau},\overline{W}^{-1},\overline{S}_1)$ establishes the result.
\end{proof}
\subsection{Stabilization}
Clearly Proposition~\ref{prop:Chap1:StabilityLMIState} provides a first condition to solve Problem~\ref{prob:Chap1:StabilityQuantizedState}. However, due to products between unknown variables,  a direct exploitation of the conditions given by Proposition~\ref{prop:Chap1:StabilityLMIState} to solve the controller design problem appears unlikely from a numerical standpoint, then further work is needed. On the other hand, if one attempts to alleviate such nonlinear terms by means of standard techniques (congruence transformations, and invertible changes of variables),  the resulting condition reveals to be still nonlinear and presenting more involved nonlinearities as trilinear terms. To somehow overcome this shortcoming, as a first step, resting on the use of the projection lemma \cite{pipeleers2009extended}), we derive an equivalent condition  to \eqref{eq:Chap1:EqProp1QuantizedState}, which is linear in the variable $P$ defining the set $\mathcal{A}$, and bilinear with respect to the controller gain and some additional variables. Such an equivalent condition is given as follows.
\begin{proposition}
	\label{prop:Chap1:ProjectionFeasibleState}
	Let $P\in\Spn,S_1,S_2\in\Dpn, K\in\R^{m\times n},\tau\in\R_{>0}$. The satisfaction of \eqref{eq:Chap1:EqProp1QuantizedState}
	is equivalent to the feasibility of 
	\begin{equation}
	\label{eq:Chap1:ProjectionState}
	\begin{bmatrix}
	-\He(X_1)&P-X_2+X_1\tr (A+BK)&X_1\tr BK\\
	\bullet&\He(X_2\tr (A+BK))+\tau P&X_2\tr BK-S_2\\
	\bullet&\bullet&-S_1-2S_2
	\end{bmatrix}<\0
	\end{equation}
	with respect to $X_1,X_2 \in\R^{n\times n}$.
\end{proposition}
\begin{proof}
From Proposition~\ref{prop:Chap1:StabilityLMIState}, note that \eqref{eq:Chap1:EqProp1QuantizedState} can be rewritten as\footnote{\textcolor{blue}{In the published version, $U\tr \mathcal{Q}U<\0$ should be replaced by $W\tr Q W<\0$. This error does not propagate and the reminder of the proof is correct. Moreover, the published version contains a typo in ${W_r^\perp}\tr$ and ${U_r^\perp}\tr$.}}  \textcolor{blue}{$W\tr Q W<\0$}, where:
	$$W=\begin{bmatrix}
	A+BK&BK\\
	\Id&\0\\
	\0&\Id
	\end{bmatrix}, Q=\begin{bmatrix}
	\0&P&\0\\
	\bullet&\tau P&-S_2\\
	\bullet&\bullet&-S_1-2S_2
	\end{bmatrix}$$
Moreover, since $S_1$ and $S_2$ are positive definite, $U\tr Q U<\0$, with $U\tr =\begin{bmatrix}\0&\0&\Id\end{bmatrix}$, is obviously satisfied. Thus, by the projection lemma; see \cite{pipeleers2009extended}, the satisfaction of \eqref{eq:Chap1:EqProp1QuantizedState}, whenever $S_1$ and $S_2$ are required to be positive definite, is equivalent to find a matrix $X$ such that
	\begin{equation}
	\label{eq:Chap1:Projection2State}
	Q+\textcolor{blue}{{W_r^\perp}\tr} X U_r^\perp+\textcolor{blue}{{U^\perp_r}\tr} X\tr W_r^\perp<\0
	\end{equation}
where ${U}_r^\perp$ and ${W}_r^\perp$ are some matrices such that ${U}_r^\perp U=0$ and ${W}_r^\perp W=0$. Now, by selecting ${U}_r^\perp=\begin{bmatrix}\Id_{2n}&\0_{2n\times p}\end{bmatrix}$ and ${W}_r^\perp=\begin{bmatrix}-\Id& A+BK& BK\end{bmatrix}$, and by partitioning $X=\begin{bmatrix}
X_1&X_2\end{bmatrix}$, where $X_1,X_2\in\mathbb{R}^{n\times n}$, from \eqref{eq:Chap1:Projection2State} one gets
	\begin{equation}
	\begin{bmatrix}
	-\He(X_1)&P-X_2+X_1\tr (A+BK)&X_1\tr BK\\
	\bullet&\He(X_2\tr (A+BK))+\tau P&X_2\tr BK-S_2\\
	\bullet&\bullet&-S_1-2S_2
	\end{bmatrix}<\0
	\end{equation}
	and this finishes the proof.
\end{proof}
\section{Numerical Issues}
\label{sec:Numerical}
The implicit objective in solving both Problem~\ref{prob:Chap1:StabilityQuantizedState} and Problem~\ref{prob:Chap1:StabilizationQuantizedState} consists of reducing the size of the set $\mathcal{A}$, with respect to a certain criterion. To this end, one can embed the conditions provided, respectively, by Proposition~\ref{prop:Chap1:StabilityLMIState} and Proposition~\ref{prop:Chap1:ProjectionFeasibleState} into suitable optimization schemes.  To this end, an adequate measure of the set $\mathcal{A}=\mathcal{E}(P)$  needs to be selected. Specifically, one needs to define a function $\mathcal{M}\colon\mathbb{R}^{n\times n}\rightarrow \mathbb{R}$ such that $\mathcal{M}(P)$ provides a convenient indication on the size of the set $\mathcal{A}$. Once $\mathcal{M}$ is defined, Proposition~\ref{prop:Chap1:StabilityLMIState} and Proposition~\ref{prop:Chap1:ProjectionFeasibleState} enable to associate the following optimization problems, respectively, to  Problem~\ref{prob:Chap1:StabilityQuantizedState}
\begin{equation}
\label{eq:Chap1:StabInst}
\begin{aligned}
& \underset{P,S_1,S_2,\tau}{\text{minimize}}
& & \mathcal{M}(P) \\
& \text{subject to}
& &S_1,S_2\in\Dpn, P\in\Spn,\tau>0\\
&&& \eqref{eq:Chap1:EqProp1QuantizedState},\eqref{eq:Chap1:EqProp2QuantizedState}.
\end{aligned}
\end{equation}
and Problem~\ref{prob:Chap1:StabilizationQuantizedState}
\begin{equation}
\label{eq:Chap1:DesignInst}
\begin{aligned}
& \underset{P,S_1,S_2,\tau,X_1,X_2,K}{\text{minimize}}
& & \mathcal{M}(P) \\
& \text{subject to}
& &S_1,S_2\in\Dpn,P\in\Spn,\tau>0\\
&&& \eqref{eq:Chap1:ProjectionState}, \eqref{eq:Chap1:EqProp2QuantizedState}.
\end{aligned}
\end{equation}

The above optimization problems are characterized by bilinear matrix inequalities constraints. Therefore, the solution to such problems is in general challenging from a numerical standpoint; see \cite{toker1995np,vanantwerp2000tutorial}.
To overcome this problem, inspired by \cite{arzelier2002iterative, ebihara2015s, vanantwerp2000tutorial}, we propose some linear matrix inequalities relaxations for problems \eqref{eq:Chap1:StabInst} and \eqref{eq:Chap1:DesignInst}.  Such relaxations enable to get suboptimal solutions to the considered optimization problems via numerically efficient algorithms. To this end, we select $\mathcal{M}$ as a convex function. Notice that $\mathcal{A}$ being an ellipsoidal set, picking $\mathcal{M}$ to be convex is not restrictive and several possibilities can be considered. Typical selections for $\mathcal{M}$ are $\log\det(P)$ or $\Tr(P)$; see \cite{Boyd} for more details.
Concerning \eqref{eq:Chap1:StabInst}, notice that when the scalar $\tau$ is fixed, such a problem is a genuine convex optimization problem over LMI constraints.  Then,  the solution  to this problem can be performed in polynomial time via, for example, interior points methods; see \cite{Boyd}. 
On the other hand, the positive scalar $\tau$ can be treated as a tuning parameter, or being selected via an iterative search. This is a typical approach pursued in the literature; see, \eg, \cite{vanantwerp2000tutorial,tarbouriech2014stability,tar:pri/cdc13}. Then \eqref{eq:Chap1:StabInst} can be efficiently solved on a computer, with only caveat to obtain a suboptimal solution. 
In particular, the proof of Proposition \ref{prop:Chap1:ExistenceStability} points out that to ensure the feasibility of \eqref{eq:Chap1:EqProp1QuantizedState} and \eqref{eq:Chap1:EqProp2QuantizedState}, one can select $\tau$ in $(0, 2\times 0.99\min \mathcal{R}(A+BK)]$, where $\mathcal{R}(A+BK)\coloneqq\{\vert\Re(\lambda)\vert\colon\lambda\in\spec(A+BK)\}$. This provides a systematic way to build an iterative algorithm for the solution to  \eqref{eq:Chap1:StabInst}.

As mentioned before, the solution to \eqref{eq:Chap1:DesignInst}, due to nonlinearities affecting condition \eqref{eq:Chap1:EqProp1QuantizedState}, is much more involved and requires a suitable strategy. In particular, inspired by \cite{arzelier2002iterative}, we propose the following iterative algorithm to derive suboptimal solutions to \eqref{eq:Chap1:DesignInst}.
\begin{algorithm}[H]
\caption{}
\begin{algorithmic}
{\small
\STATE{\bf Input:} Matrices $A,B$, scalar $\Delta>0$, $\overline{K}\in\R^{m\times n}$ such that $A+B\overline{K}$ is Hurwitz, a convex objective function $\mathcal{M}$, and a desired tolerance $\rho>0$.
\STATE{\bf Initialization:} 
			Select $\overline{\tau}=2\times 0.99\min ,\mathcal{R}(A+B\overline{K})$ and build a grid of positive values $\mathcal{G}_{\tau}$ such that $\max \mathcal{G}_{\tau}=\overline{\tau}$ (this ensures the feasibility of the resulting optimization problems).
			\STATE{\bf Iteration}
			
			\smallskip
			\noindent{\em Step 1:} 
			Given $\overline{K}$ from the previous step, solve the following convex optimization problem over LMIs by selecting $\tau$ over $\mathcal{G}_{\tau}$.
			\begin{equation}
			\label{eq:Chap1:AlgoOpt0}
			\begin{aligned}
			& \underset{S_1, S_2, P, X_1, X_2}{\minimize}
			& & \mathcal{M}(P) \\
			& \text{s.t.}
			&&S_1,S_2\in\Dpn,P\in\Spn\\
			&&&\left[\begin{smallmatrix}
			-\He(X_1)&P-X_2+X_1\tr (A+B\overline{K})&X_1\tr B\overline{K}\\
			\bullet&\He(X_2\tr (A+B\overline{K}))+\tau P&X_2\tr B\overline{K}-S_2\\
			\bullet&\bullet&-S_1-2S_2
			\end{smallmatrix}\right]<\0\\
			&&&\Tr(S_1)\Delta^2-\tau\leq 0
			\end{aligned}
			\end{equation}
			Pick the suboptimal solution obtained and set  $\overline{X}_1=X_1, \overline{X}_2=X_2$ for the next step.
			
			\smallskip
			\noindent{\em Step 2:} 
			Given $\overline{X}_1,\overline{X}_2$ from the previous step, solve the following convex optimization problem over LMIs by selecting $\tau$ over  $\mathcal{G}_{\tau}$.
			\begin{equation}
			\label{eq:Chap1:AlgoOpt1}
			\begin{aligned}
			& \underset{S_1, S_2, P, K}{\minimize}
			& & \mathcal{M}(P) \\
			& \text{s.t.}
			&&S_1,S_2\in\Dpn,P\in\Spn\\
			&&&\left[\begin{smallmatrix}
			-\He(\overline{X}_1)&P-\overline{X}_2+\overline{X}\tr _1(A+BK)&\overline{X}\tr _1BK\\
			\bullet&\He(\overline{X}\tr _2(A+BK))+\tau P&\overline{X}\tr _2BK-S_2\\
			\bullet&\bullet&-S_1-2S_2
			\end{smallmatrix}\right]<\0\\
			&&&\Tr(S_1)\Delta^2-\tau\leq 0.
			\end{aligned}
			\end{equation}
			Set $\overline{K}=K$,  for the next step.
			\STATE Determine the closed-loop matrix $A+B\overline{K}$ and set $\overline{\tau}=2\times 0.99\min \mathcal{R}(A+B\overline{K})$. Build a grid of positive values $\mathcal{G}_{\tau}$ such that $\bar{\tau}=\max\mathcal{G}_{\tau}$, and $\tau^{\star}\in\mathcal{G}_{\tau}$ (notice that necessarily $\tau^{\star}\leq \bar{\tau}$. Including $\tau^{\star}$ in $\mathcal{G}_{\tau}$ ensures the feasibility at the next step).
			
			\STATE {\bf Until} $\mathcal{M}$ does not decrease below $\rho$ over three consecutive steps.
			\STATE {\bf Output:} $\overline{K}, \overline{P}$.}
	\end{algorithmic}
	\label{Algo:Chap1:StateDesign}
	\end{algorithm}
\begin{remark}
Proposition \ref{prop:Chap1:ProjectionFeasibleState} plays a determinant role in the development of the above exposed algorithm. In fact, the introduction of the slack variables $X_1,X_2$ enables to treat $P$ as a decision variable at each step of the algorithm, without adding any additional conservatism; recall that the feasibility of \eqref{eq:Chap1:Projection2State} is equivalent to the one of \eqref{eq:Chap1:EqProp1QuantizedState}. Notice that, by exploiting directly Proposition \ref{prop:Chap1:StabilityLMIState}, due to the bilinear terms involving the matrix $P$ and the controller gain $K$, to retrace the strategy proposed in Algorithm~\ref{Algo:Chap1:StateDesign},  one would need to alternatively fix either $P$ or $K$, preventing from treating $P$ as a decision variable at each step.  This obviously has a dramatic impact on the achievable suboptimal solutions to \eqref{eq:Chap1:DesignInst}.
\end{remark}
The above algorithm presents some interesting properties that render its utilization in practice quite convenient. Such properties are stated in the following result.
\begin{proposition}
	Algorithm~\ref{Algo:Chap1:StateDesign} possess the following properties:
	\begin{itemize}
		\item[(a)]  For any given input such that $(A,B)$ is stabilizable, Algorithm~\ref{Algo:Chap1:StateDesign} always returns a suboptimal solution to Problem~\ref{prob:Chap1:StabilizationQuantizedState}
		\item [(b)] For any given input such that $(A,B)$ is stabilizable, Algorithm~\ref{Algo:Chap1:StateDesign} always terminates in a finite number of steps. 
	\end{itemize}
\end{proposition}
\begin{proof}
	To establish (a), one needs to show that Algorithm~\ref{Algo:Chap1:StateDesign} never runs into a deadlock. To this end, notice that at each iteration, both \eqref{eq:Chap1:AlgoOpt0} and \eqref{eq:Chap1:AlgoOpt1} are always feasible. Indeed, during the first iteration, since $A_{cl}$ is Hurwitz, the feasibility of \eqref{eq:Chap1:AlgoOpt0} is ensured by Proposition \ref{prop:Chap1:ExistenceStability}. To see that also at each other iteration the considered optimization problem are always feasible, consider the following arguments.
	For the $j-th$ iteration, denote the value of the matrix $P$, respectively, at the exit of step 1 and of step 2 as $\overline{P}_{j}^{(1)}$ and $\overline{P}_{j}^{(2)}$.
	
	[From step 1 to step 2]  Obviously step 2 is always feasible, indeed keeping the same gain $\overline{K}$  from the previous step yields a feasible solution, and moreover $\mathcal{M}_s(\overline{P}_{j}^{(2)})\leq \mathcal{M}_s(\overline{P}_{j}^{(1)})$. 
	
	[From step 2 to step 1]  The  feasibility of \eqref{eq:Chap1:AlgoOpt0} is ensured by following the same arguments performed in the previous case. Then,  Algorithm~\ref{Algo:Chap1:StateDesign} never terminates due to infeasibility of ,  \ie,  (a) is proven. 
	
	To show (b), assume that $\mathcal{M}_s(P)\geq 0$ over the feasible set of \eqref{eq:Chap1:AlgoOpt1} (this assumption can be fulfilled without any difficulty by relying on standard convex criteria or by slightly modifying the stopping criterion to meet this assumption). Then, as shown in the proof of (a), since $\mathcal{M}_s(\overline{P})_{j+1}^{(1)})\leq  \mathcal{M}_s(\overline{P}_{j}^{(2)})$, it follows that the sequence $\{\mathcal{M}_s(\overline{P}_{j}^{(2)})\}_{j=1}^\infty$ converges when $j$ approaches infinity. 
	Thus, for any positive $\rho$, there exists $\bar{j}\in\mathbb{N}$ such that for each $j>\bar{j}$ one has $\vert\mathcal{M}_s(\overline{P}_{j}^{(2)})-\mathcal{M}_s(\overline{P}_{j+1}^{(2)})\vert<\rho$. Hence, (b) is proven.
\end{proof}
\section{Numerical Examples} 
\label{sec:Examples}
\begin{example}
\label{ex:chap1:StateBullo}
Consider again the case analyzed in Example~\ref{ex:Example0}. For this example, we shown that the set $\mathcal{S}$ is composed by Krasovskii equilibria for the closed-loop system. This implies that the smallest convex uniformly globally attractive compact set containing the origin one may expect to determine for the closed-loop system is 
$\Co\mathcal{S}$. This provides in this example a way to (exactly) measure the conservatism of our approach.
	In particular, by selecting as a convex criterion $\mathcal{M}_s(P)=-\log\det(P)$ and by performing a grid search for $\tau$ on $(0,2\times 0.99\min\mathcal{R}(A+BK)]$ one gets
	$$
	P=\begin{bmatrix}
	2912.617612709  &   -2912.612579926\\
	-2912.612579926  &    2912.609874515
	\end{bmatrix}
	$$
	\figurename~\ref{Fig:Chap1:3BulloLib} reports the set $\mathcal{E}(P)$, along with some closed-loop solutions.
	Notice that 
	$$\spec(P)\approx\{0.001163683, 5825.22\}$$
	ensuring that $P>\0$. The large difference between such eigenvalues is due to the shape of the set $\mathcal{E}(P)$  represented in \figurename~\ref{Fig:Chap1:3BulloLib}, which nearly corresponds to a segment. In particular, the largest semi axis of the set $\mathcal{A}$ is approximately given by the vector $29.314\times (0.707106, 0.707107)$, which nearly corresponds to the segment joining the points $(-20.72,-20.72)$ and $(20.72,20.72)$. Namely, $\mathcal{A}$ coincides approximately with $\Co\mathcal{S}$, showing the effectiveness of the proposed methodology, at least in this example. 
	\begin{figure}
		\centering
		\psfrag{x1}[c][][1]{$x_1$}
		\psfrag{x2}[c][1]{$x_2$}
		\includegraphics[scale=0.8]{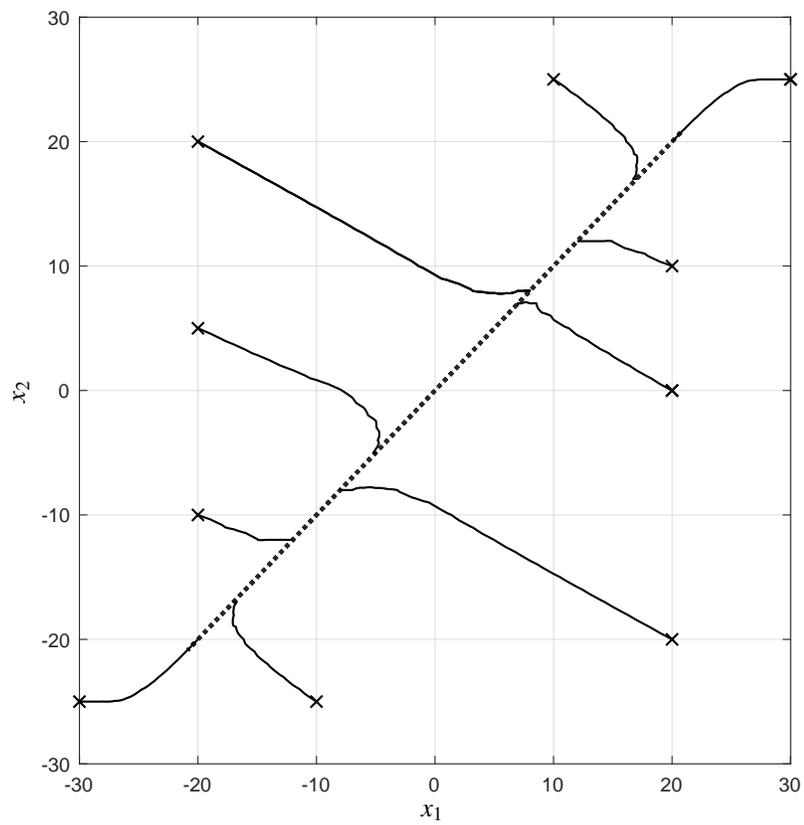}
		\caption{The evolution of the closed-loop system from different initial conditions and the closed-loop attractor $\mathcal{E}(P)$ (dotted line).}
		\label{Fig:Chap1:3BulloLib}
	\end{figure}
Now we compare our results to similar results in \cite{lib/book2003}. Indeed, by relying on the connections exploited in this paper between finite time convergence, strong forward invariance, and Lyapunov stability, a (non-constructive) methodology to get an attractor for the closed-loop system can be obtained by relying on the results in \cite{lib/book2003}. In particular, let $\widetilde{P}\in\Spn$ be a solution to the following Lyapunov equality
		$$
		\He((A+BK)\tr \widetilde{P})=-Q
		$$
		where $Q\in\Spn$. Define, 
		$$
		\Theta_x=2\frac{\vert\widetilde{P}BK\vert}{\lambda_{\min}(Q)}, \rho=\sqrt{2}\Delta\Theta_x
		$$
		Then in \cite{lib/book2003} it is shown that
		$$\mathcal{E}\left(\frac{\widetilde{P}}{\lambda_{\max}(P)\rho^2}\right)$$ 
		is an attractor for the closed-loop system. This methodology is non constructive in the sense that no clear guideline exists to set the matrix $Q$ so as to reduce the size of the attractor. 
		To evaluate the benefit of our methodology in this example, we compare the set $\mathcal{E}(P)$ with a batch of $100$ attractors generated as illustrated above for which the matrix $Q$ is selected randomly, the results are reported in \figurename~\ref{Fig:Chap1:3BulloLib2}. Numerical results show that our methodology provides the tightest attractor in this example. Moreover, \figurename~\ref{Fig:Chap1:3BulloLib2} points out that the resulting attractor strongly depends on the selection of the matrix $Q$.
\begin{figure}
\psfrag{x1}[][][0.8]{$x_1$}
\psfrag{x2}[][][0.8]{$x_2$}
\includegraphics[scale=0.4]{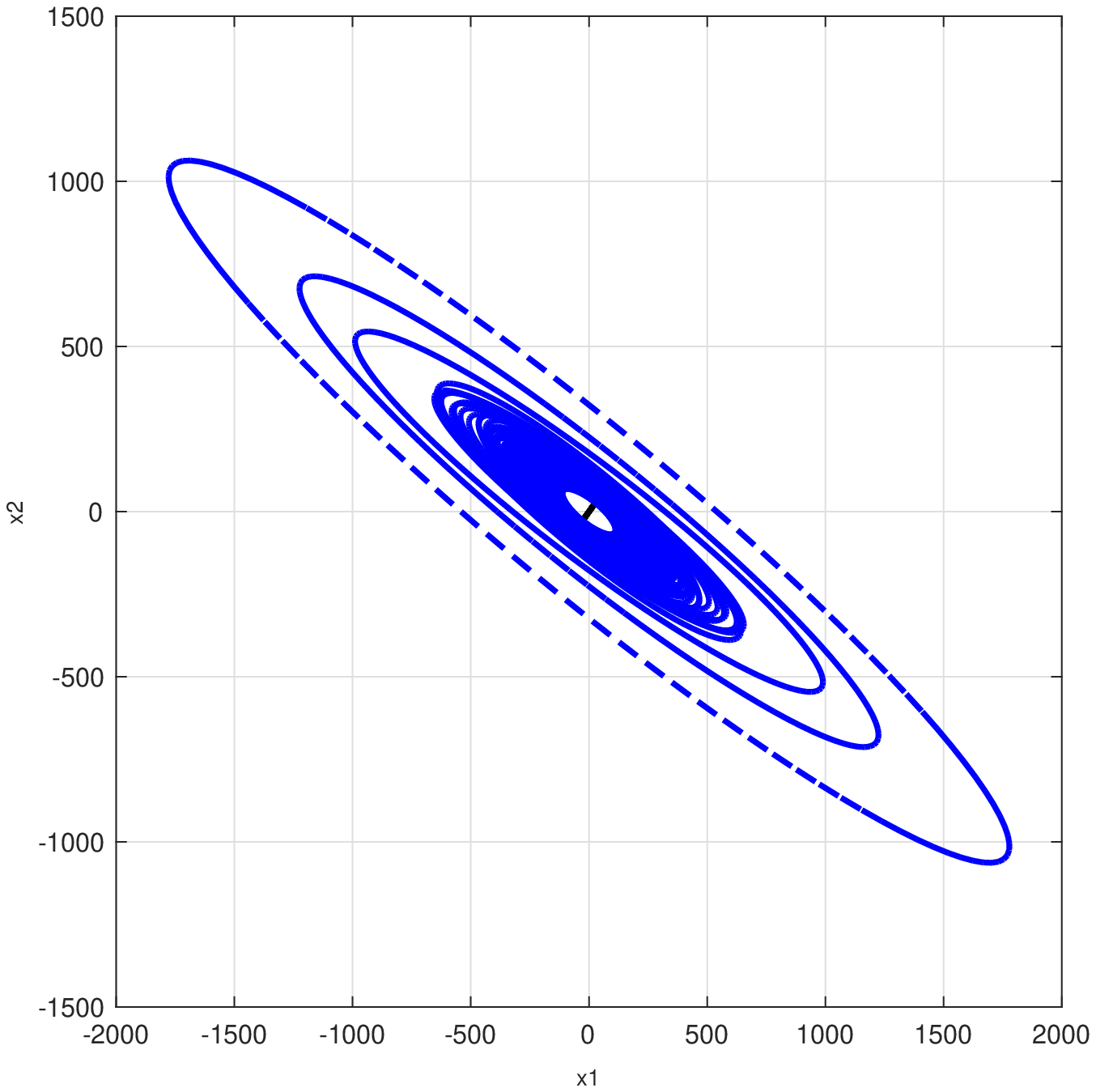}
\includegraphics[scale=0.4]{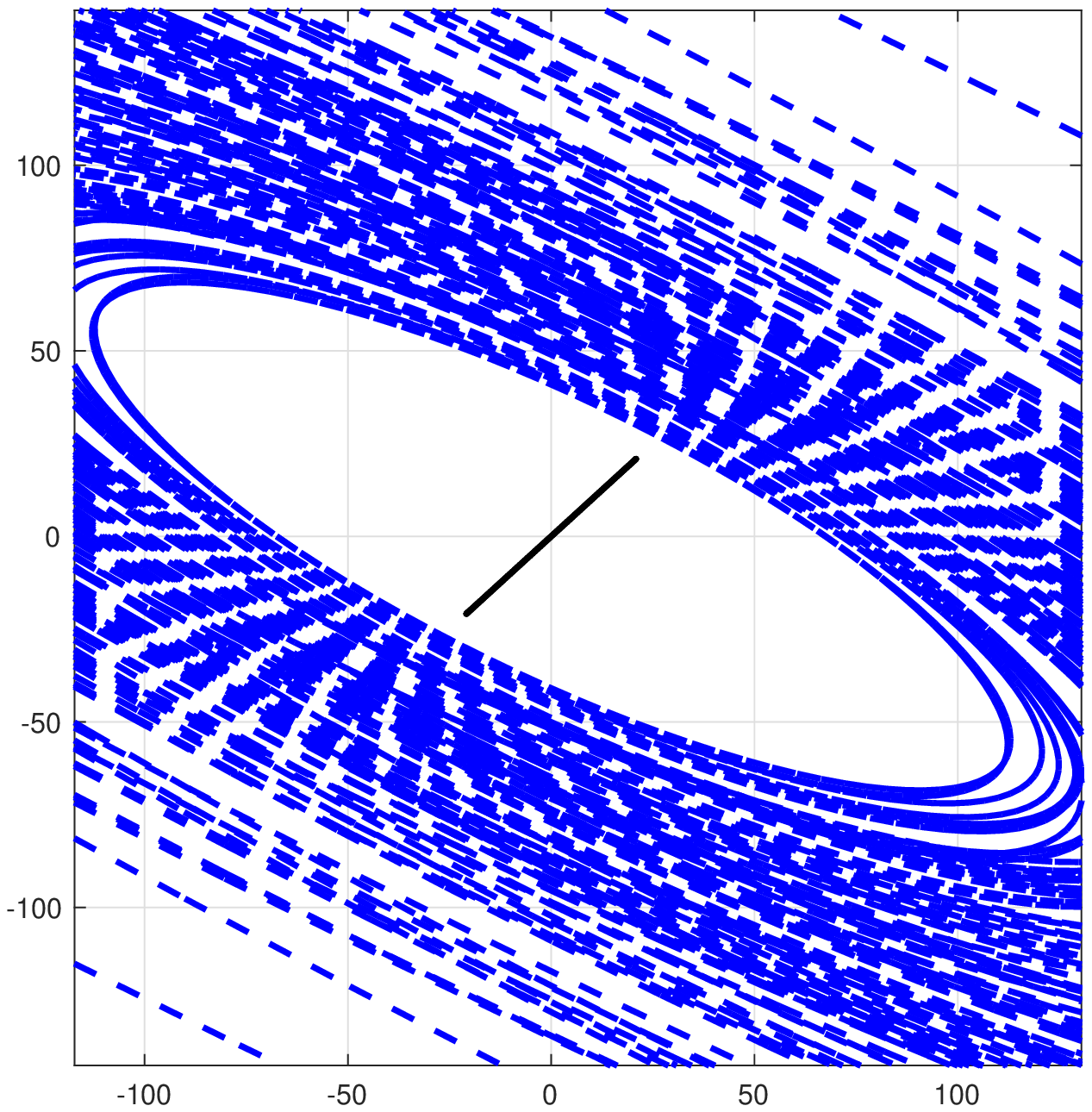}
\caption{Closed-loop attractors: $\mathcal{E}(P)$ (solid line) and $\mathcal{E}\left(\frac{\widetilde{P}}{\lambda_{\max}(\widetilde{P})\rho^2}\right)$ obtained with different matrices $Q$ (dashed line), a close-up in the right picture.}
\label{Fig:Chap1:3BulloLib2}
\end{figure}
\end{example}
\begin{example}(A multi-input plant)
\label{ex:Chap1:3dState}
Consider the example from \cite{amato2009stabilization} for which
$$A=\begin{bmatrix}
-0.5 & 1.5 & 4\\ 4.3 & 6 & 5\\ 3.2 & 6.8 & 7.2
\end{bmatrix}, B=\begin{bmatrix}-0.7 & -1.3\\ 0 & -4.3\\ 0.8 & -1.5\end{bmatrix},$$ 
and assume that the measured state is quantized via the uniform quantizer \eqref{eq:Chap1:Uniformquantizer} with $\Delta=(0.5, 0.5)$. We solve \eqref{eq:Chap1:DesignInst} using Algorithm~\ref{Algo:Chap1:StateDesign}. As a convex criterion to minimize, we proceed as in \cite{FerranteAutomatica2015}. In particular, let $N\in\Spn$ be an additional decision variable, by enforcing
\begin{equation}
\begin{bmatrix}
N&\Id\\
\bullet&P
\end{bmatrix}\geq\0
\label{eq:N_opt}
\end{equation}
we minimize $\Tr(N)$, which, thanks to the above constraint, indirectly entails the minimization of $\Tr(P^{-1})$. The selection of such a criterion allows to solve each optimization problem occurring in Algorithm~\ref{Algo:Chap1:StateDesign} via standard semidefinite programming (SDP) solvers. 

In this example, we test Algorithm~\ref{Algo:Chap1:StateDesign} by using three different stabilizing gains for its initialization. The first one
	$$
	K_1=\begin{bmatrix}
	0.0380 & 0.1751& -0.8551\\ 3.8514 & 3.8400 & 9.5510
	\end{bmatrix}
	$$
is borrowed directly from  \cite{amato2009stabilization}. The second one
	$$
	K_2=\begin{bmatrix}
	-0.71  & 1.9 &  -27\\
	4.3&  4.1&  4.3
	\end{bmatrix}
	$$
comes from \cite{FerranteThesis}, and finally the third one,
	$$
	K_3=\begin{bmatrix}
	-0.11527  &  -0.28207  &    -1.2449\\
	2.4835    &   4.2519     &  6.2107
	\end{bmatrix}$$
is the gain issued from the solution to an LQR problem on the pair $(A,B)$, with $Q=\Id_3$, and $R=\Id_2$. For all these three initializations, the tolerance for the algorithm is $\rho=10^{-4}$.
	\figurename~\ref{Fig:Chap1:3DStateIiterations} shows the evolution of  $\Tr(N)$ over the number of iterations for the three proposed initializations. Surprisingly, although Algorithm~\ref{Algo:Chap1:StateDesign} does not ensure convergence toward the optimal solution and the initialization are quite different of each other, the algorithm provides three solutions giving nearly the same value of the objective. This shows, at least in this specific example, that the initialization stage is not excessively crucial, though it may impact the computational burden: the number of iterations might increase depending on the initialization, \eg,  for the third initialization the number of iterations is almost twice  as much as the number of iterations occurring for the second initialization. 
	\begin{figure}[ht!]
		\centering
		\psfrag{x}[][][1]{\# of iterations}

		\includegraphics[scale=0.8]{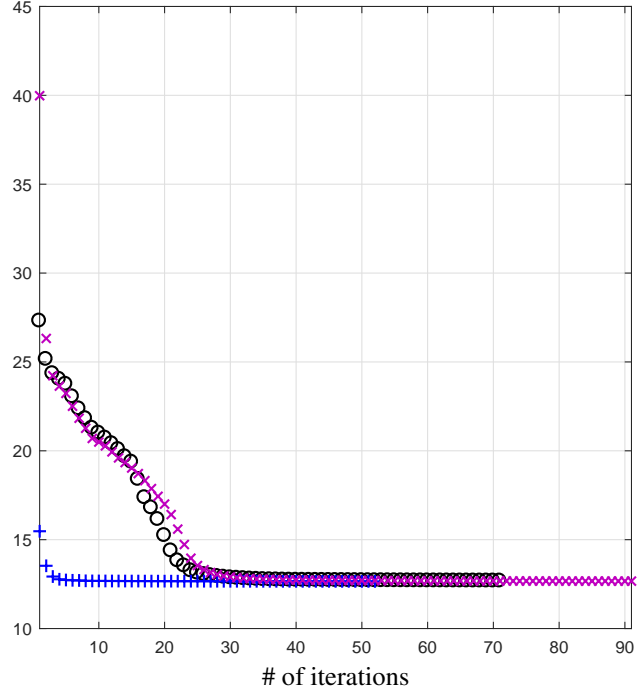}
		\caption[]{The value of $\Tr(N)$ versus the number of iterations. First initialization ($\times$), second initialization ($+$), and third initialization ($\bigO$).}
		\label{Fig:Chap1:3DStateIiterations}
	\end{figure}  
	In \tablename~\ref{tab:Chap1:3dExample}, the different outputs of the algorithm are reported for the three considered initializations. 
	As shown in \tablename~\ref{tab:Chap1:3dExample}, the first and the third initialization provide quite similar results also in terms of controller gain $K$ and the matrix $P$ defining the set $\mathcal{A}=\mathcal{E}(P)$ solving the controller design problem.
	
	To assess the improvement produced by Algorithm~\ref{Algo:Chap1:StateDesign}, by relying on Proposition~\ref{prop:Chap1:StabilityLMIState}, for each of the considered initialization, we perform an analysis stage of the closed-loop system obtained by adopting as a control gain the one chosen to initialize Algorithm~\ref{Algo:Chap1:StateDesign}. The result of such analysis consists, for each initialization, of a matrix $P_a\in\Spn$ such that the set $\mathcal{E}(P_a)$ is the solution to the stability analysis problem associated to the considered closed-loop system.
	\begin{table}[h!]
		\centering
		{\small\begin{tabular}{ l l l l l}
				\hline
				Init. gain & $\Tr(N)$&$K$ & $P$& \# Iter. \\ \hline\\
				$K_1$&$12.673$& $\begin{bmatrix}-3.3894  & -0.37076  &   -25.912\\ 4.4094  &0.57293 & 16.594\end{bmatrix}$& $\begin{bmatrix}0.98947 & 0.29953 & -0.53257\\ 0.29953 & 0.36835 & -0.60327\\ -0.53257 & -0.60327 & 1.448 \end{bmatrix}$&71 \\\\
				$K_2$& $12.653 $ &$\begin{bmatrix}-8.4812   &   -0.1307  & -30.799\\ 6.3613 & 1.4677 & 13.504\end{bmatrix}$ &$\begin{bmatrix}
				0.96943 & 0.30148 & -0.52565\\ 0.30148 & 0.37372 & -0.60606\\ -0.52565 & -0.60606 & 1.4383
				\end{bmatrix}$&52\\\\
				$K_3$ & $12.669$ &$\begin{bmatrix} -3.4714   &  -0.39353   &   -25.839\\ 4.3721 &  0.60222  &  16.29\end{bmatrix}$ &$\begin{bmatrix}0.58983 & 0.17519 & -0.12874\\ 0.17519 & 0.41502 & -0.54817\\ -0.12874 & -0.54817 & 1.123 \end{bmatrix}$&91\\\\
				\hline
		\end{tabular}}
		\caption{The different outputs of Algorithm \ref{Algo:Chap1:StateDesign} for the three different initializations.}
		\label{tab:Chap1:3dExample}
	\end{table}
	In particular, since in this example the measure of the set $\mathcal{A}=\mathcal{E}(P)$ chosen to design the controller is related to $\Tr(P^{-1})$, given a stabilizing gain $K$, to perform the above mentioned analysis stage, we solve the following optimization problem
	$$
	\begin{aligned}
	& \underset{P,S_1,S_2,\tau}{\text{minimize}}
	& & \Tr(N) \\
	& \text{subject to}
	& &S_1, S_2\in\Dpn, P\in\Spn, N\in\Spn, \tau>0, \eqref{eq:Chap1:EqProp1QuantizedState},\eqref{eq:Chap1:EqProp2QuantizedState}\\
	&&&\begin{bmatrix}
	N&\Id\\
	\bullet&P
	\end{bmatrix}\geq\0
	\end{aligned}
	$$
	As illustrated previously, to overcome the nonlinearity introduced by the product $\tau P$, we perform a grid search for $\tau$ over the interval $(0, 2\times 0.99\min \mathcal{R}(A+BK)]$. By performing such an analysis for each of the considered initialization gains, we obtain the results shown in \tablename~\ref{tab:Chap1:3dExample2}. Numerical results show the effectiveness of the proposed design strategy in reducing the size of the closed-loop attractor $\mathcal{A}$. The shrinkage of the set $\mathcal{A}$ clearly emerges from \figurename~\ref{fig:plots}. 
	\begin{table}[h!]
		\centering
		\begin{tabular}{l l l l}
			\hline
			Initialization gain&$\Tr(P^{-1}_a)$ & $\Tr(P^{-1})$&Improvement \%\\ \hline
			$K_1$&$30.3394$&$12.6719$&$58$\\
			$K_2$&$18.0336$&$12.6682$&$30$ \\
			$K_3$&$73.4109$&$12.6532$&$83$\\
			\hline
		\end{tabular}
		\caption{Size of the set $\mathcal{A}$ for the system in closed loop with the initialization gain and the designed gain.}
		\label{tab:Chap1:3dExample2}
	\end{table}
	\begin{figure}[!htbp]
		\centering
		\psfrag{x1}[][][0.5]{$x_1$}
		\psfrag{x2}[][][0.5]{$x_2$}
		\psfrag{x3}[][][0.5]{$x_3$}
		\includegraphics[scale=0.22]{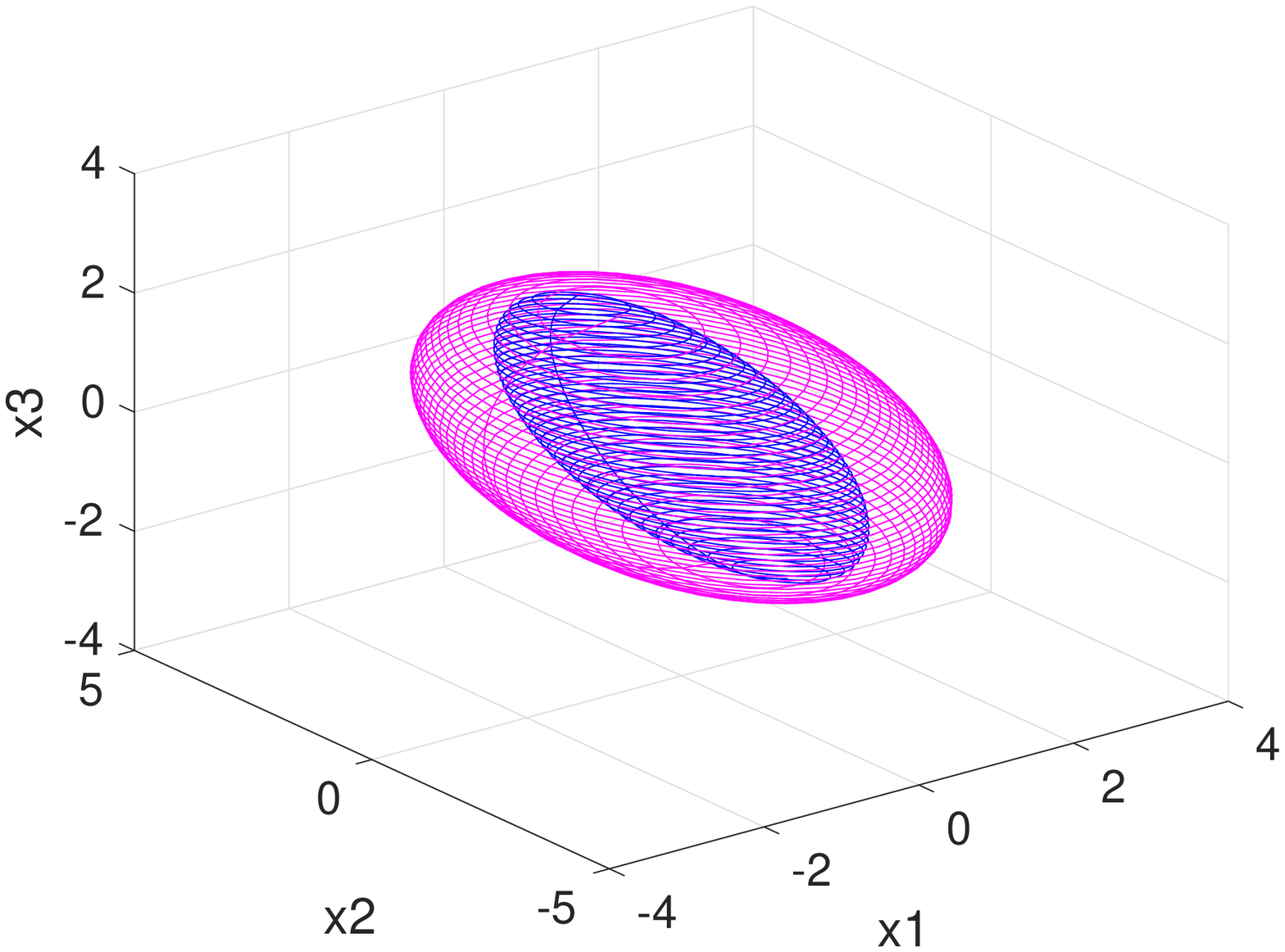}
		\includegraphics[scale=0.22]{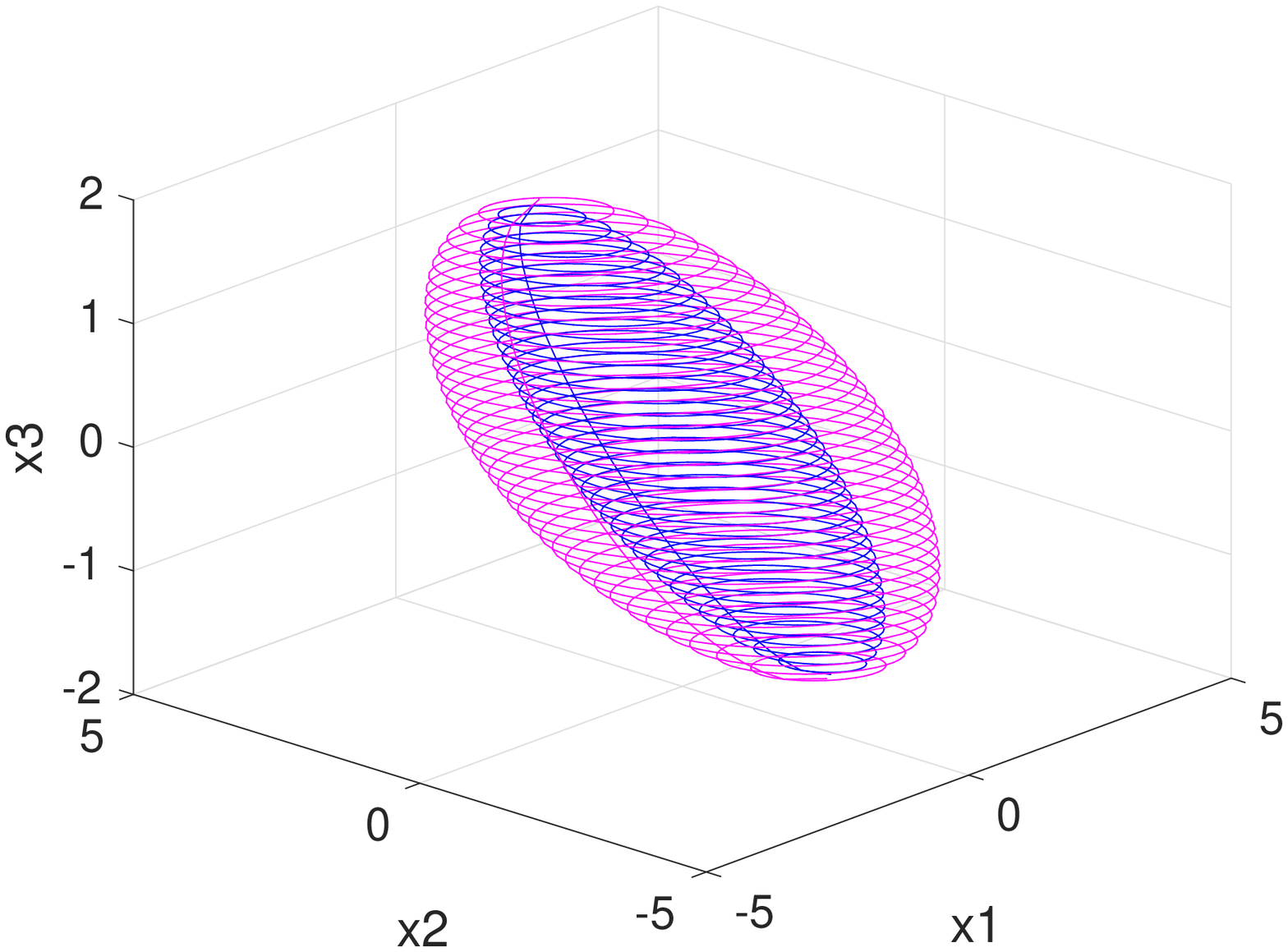}
		\includegraphics[scale=0.22]{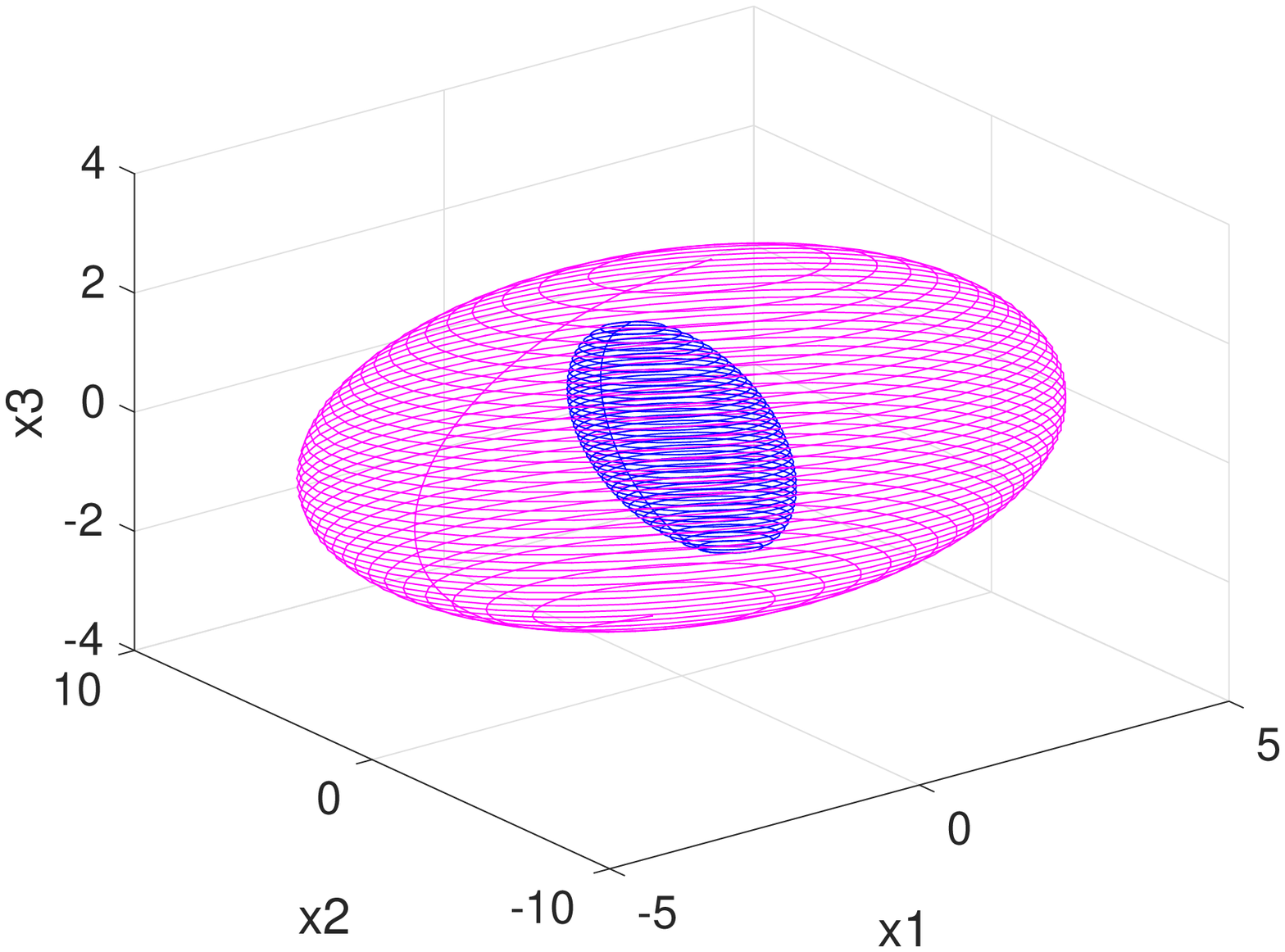}
		\caption{The sets $\mathcal{E}(P_a)$ (smaller ellipsoid) and  $\mathcal{E}(P)$ (larger ellipsoid) obtained with the different initialization gains: $K_1$ (left), $K_2$ (center), and $K_3$ (right)}
\label{fig:plots}
\end{figure}
Before concluding this example, we want to offer a last comparison aimed at showing  
the advantage of the proposed relaxation strategy in tackling bilinear matrix inequality (BMI) constraints over the use of currently available bilinear matrix inequalities solvers.  To this end, we tested two specific BMI solvers: \textbf{BMIBNB} \cite{yalmip} and \textbf{PENBMI} \cite{kocvara2005penbmi}.  BMIBNB is a global solver based on branch-and-bound and convex relaxation while PENBMI is a local solver (then needing an initial point) based  on a generalization of the \emph{Augmented Lagrangian method} proposed in \cite{ben1997penalty}; we invite to see \cite{kocvara2005penbmi} for further details on PENBMI. 
	Numerical tests reveal that, in this example, both BMIBNB and PENBMI (initialized with the same gains considered above) return infeasibility. On the other hand, it is interesting to observe that when PENBMI is initialized as suggested in the initialization step of Algorithm \ref{Algo:Chap1:StateDesign}, it is able to determine a feasible solution. In particular, for each of the three initialization gains considered above, PENBMI returns:
		$$
		\begin{aligned}
		&P=\begin{bmatrix}
		0.6720   & 0.2130  & -0.2283\\
		0.2130  &  0.4045 &  -0.5677\\
		-0.2283  & -0.5677 &  1.2078
		\end{bmatrix}\\
		&K=\begin{bmatrix}
		-6.9421 &  -0.2704 & -28.8272\\
		5.9310 &   1.2306  & 14.1007
		\end{bmatrix}
		\end{aligned}
		$$
		for which one has $\Tr(P^{-1})=12.6465$, \ie,  the improvement with respect to our algorithm is about $0.05\%$. 
		These results show that the initialization step proposed in our algorithm allows to determine, at least in this example, a good estimate of a (potentially global) minimum of the considered optimization problem. It also suggests that the relaxation technique we conceived is effective in the solution of the control design problem addressed in this 
paper when compared with general purpose BMI solvers.
\end{example}
\begin{example}
To illustrate the applicability of our methodology on a practically relevant example. In particular, consider the planar path following problem in \cite{Moarref2014}, where the objective is to control a unicycle robot to track the line $y=0$. The linearized model of the robot around the origin is given by
$$
\begin{bmatrix}
\dot{x}_1\\
\dot{x}_2\\
\dot{x}_3
\end{bmatrix}=\begin{bmatrix}
0&1&0\\
0&-0.01&0\\
1&0&0
\end{bmatrix}\begin{bmatrix}
x_1\\
x_2\\
x_3
\end{bmatrix}+\begin{bmatrix}
0\\
1\\
0
\end{bmatrix}u
$$
where $x_1$ represents the heading angle, $x_2$ represents the angular speed around the yaw axis, i.e., $x_2=\dot{x}_1$, $x_3$ represents the distance to the line $y=0$, and $u$ is the torque input abound the yaw axis. 
We consider the scenario schematically represented in \figurename~\ref{fig:scenario}, in which the state of the robot is measured through a camera and broadcast to the robot via a wireless network. Such measurements are used by a controller colocated with the robot to achieve asymptotic stabilization of the origin. As in \cite{cer:dep:fra/automatica2011}, to allow for finite-bandwidth communication, we consider that the state of the robot is quantized according to the function \eqref{eq:Chap1:Uniformquantizer}, so that transmission occurs only when the state transitions from one quantization region to another.
Within the considered setting, if one assumes that transmission delays are negligible, then the considered setup can be addressed with the tools presented in this paper.    
\begin{figure}[ht!]
		\centering
		\psfrag{y}[][][1]{$x_3$}
		\psfrag{t}[][][1]{$x_1$}
		\includegraphics[scale=0.8]{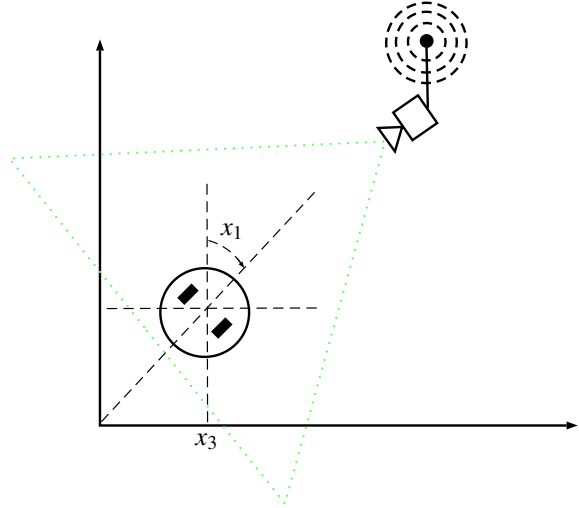}
		\caption{The considered scenario.}
		\label{fig:scenario}
	\end{figure} 
For the purpose of this example, we assume that the quantization error bound vector is selected as $\Delta=\left(\frac{\pi}{12}, 2, 0.01\right)$. Following the same lines as in  Example~\ref{ex:Chap1:3dState}, we solve \eqref{eq:Chap1:DesignInst} using Algorithm~\ref{Algo:Chap1:StateDesign} in which we minimize $\Tr(N)$, where $N$ is subject to \eqref{eq:N_opt}. In particular, by picking as a tolerance $\rho=10^{-2}$ and as an initialization gain\footnote{The initialization gain is the solution to the LQR problem on the pair $(A,B)$ with $Q=\Id$ and $R=1$} 
$$
K_0=\begin{bmatrix}
-2.4142  &  -2.4042  &  -1
\end{bmatrix}
$$ 
Algorithm~\ref{Algo:Chap1:StateDesign} terminates in $19$ iterations and returns 
$$
\begin{aligned}
&K=\begin{bmatrix}-10.7991 &  -3.7784 &  -6.6328\end{bmatrix}\\
&P=\begin{bmatrix}
0.3293  &  0.0418 &   0.1916\\
0.0418  &  0.0419   & 0.0481\\
0.1916    &0.0481   & 0.3857
\end{bmatrix}
\end{aligned}
$$
\figurename~\ref{fig:objective_iterations} shows the evolution of  $\Tr(N)$ over the number of iterations for the considered initialization. The figure shows that Algorithm~\ref{Algo:Chap1:StateDesign} is able to move sufficiently away from the initial point leading to a decrease about $39\%$ of the objective function.
\begin{figure}[h!]
\centering
\psfrag{x}[][][1]{\# of iterations}
\includegraphics[scale=0.6]{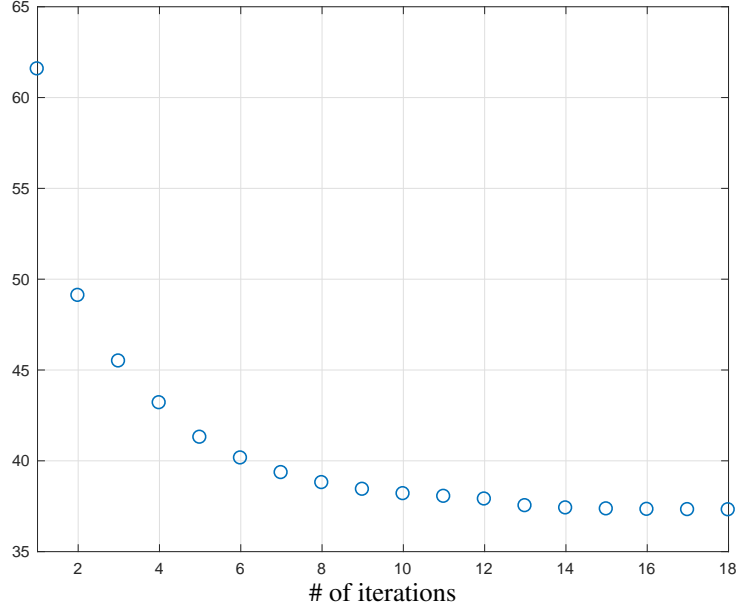}
\caption{The value of $\Tr(N)$ versus the number of iterations.}
\label{fig:objective_iterations}
\end{figure}  
\figurename~\ref{fig:robot_ellipse} depicts some closed-loop trajectories obtained for the designed gain along with the ellipsoidal attractor $\mathcal{E}(P)$. The figure  suggests that closed-loop trajectories converge to a limit cycle contained in $\mathcal{E}(P)$.   
\begin{figure}[ht!]
\centering
\psfrag{x}[][][1]{$x_1$}
\psfrag{y}[][][1]{$x_2$}
\psfrag{z}[][][1]{$x_3$}
\includegraphics[scale=0.7]{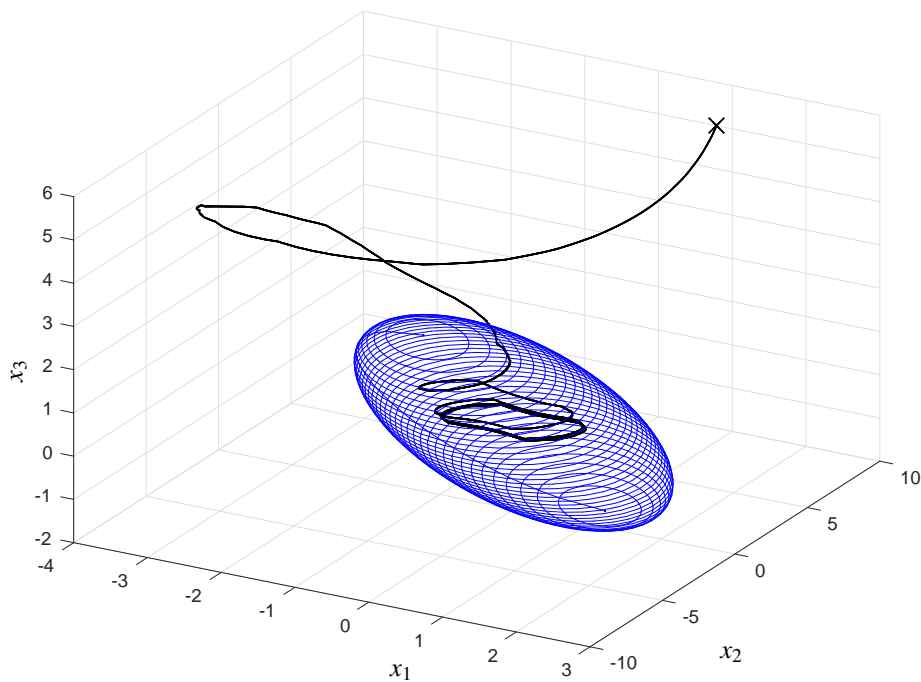}
\caption{The ellipsoidal attractor $\mathcal{E}(P)$ (blue) and a the closed-loop trajectory  from $x(0)=(\pi/4, 10, 5)$.}
\label{fig:robot_ellipse}
\end{figure}  
To emphasize the improvement of our design methodology in reducing the effect of sensor quantization when compared to the initialization gain, in \figurename~\ref{fig:comparison_robot} we report the evolution of the closed-loop system corresponding to the initialization gain $K_0$ and to the designed gain $K$. The figure clearly points out that the gain returned by our algorithm provides an improved behavior with respect to $K_0$.
\begin{figure}[ht!]
\centering
\psfrag{t}[][][1]{$t$}
\psfrag{y1}[][][1]{$x_1$}  
\psfrag{y2}[][][1]{$x_2$}
\psfrag{y3}[][][1]{$x_3$}
\includegraphics[scale=0.7]{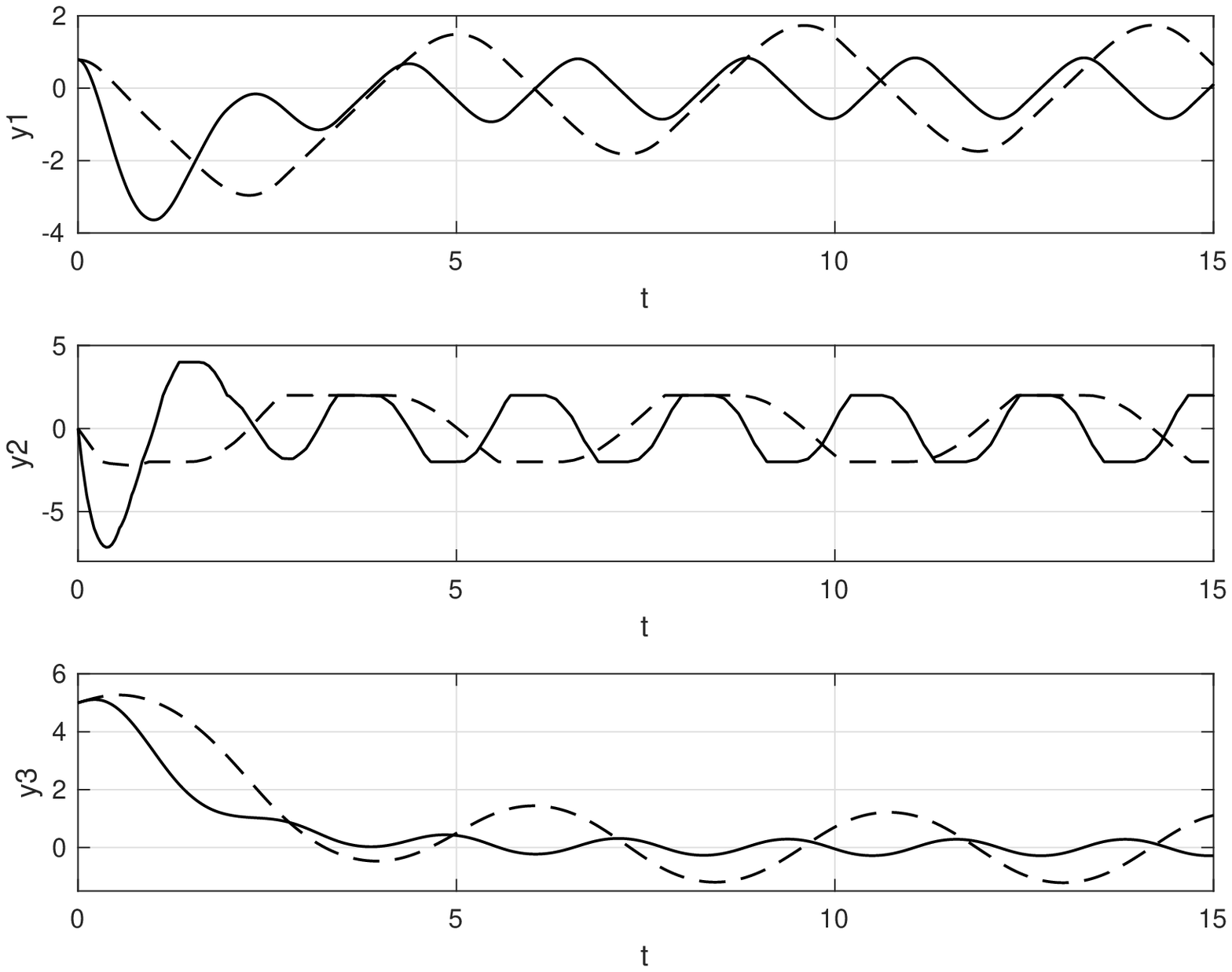}
\caption{Comparison between closed-loop trajectories from $x(0)=(\pi/4, 0, 5)$ for different gains: initialization gain $K_0$ (dashed line) and designed gain $K$ (solid line).}
\label{fig:comparison_robot}
\end{figure}  
\end{example}
\section{Conclusion}
In this paper we addressed stability analysis and stabilization for linear state feedback control systems in the presence of quantized sensor. The analysis is carried out by embedding tools from the literature of discontinuous right-hand side differential equations, as the notion of Krasovskii solutions, in a semidefinite programming setup. The outcome of this approach consists of constructive results to perform the analysis and the design of quantized control systems, while accounting for the discontinuous nature of the closed-loop system.
More precisely, thanks to the novel sector conditions for the uniform quantizer presented in \cite{Ferrante2014ECC},  we proposed sufficient conditions in the form of matrix inequalities that can be used to determine a compact set $\mathcal{A}$ containing the origin that is uniformly globally asymptotically stable for the closed-loop system. Moreover, we shown that the proposed approach is lossless in the sense that if the quantization-free closed-loop system is asymptotically stable, then the sufficient conditions we formulated are always feasible.
Building on the conditions we proposed, some algorithms based on convex optimization over LMIs are devised to effectively solve the considered problems, while providing (sub)optimal solutions with respect to convenient objectives. Numerical comparisons show that the algorithm we conceived is competitive with respect to commercially available BMI solvers.
The effectiveness of the general methodology is shown in some examples. These examples, not only provide a benchmark to test the proposed apparatus from a numerical standpoint, but also point out the complexity hidden behind quantized control systems. 
\clearpage
\bibliographystyle{plain}
\bibliography{biblio}



\end{document}